\documentclass[11pt]{amsart}
\usepackage[margin=1in]{geometry}
\usepackage{hyperref}
\usepackage{amsmath}
\usepackage{amssymb}
\usepackage{amsthm}
\usepackage{graphicx}
\usepackage{graphics}
\usepackage{tikz-cd}
\usepackage{cite}
\usepackage[mathscr]{euscript}
\usepackage{shuffle}
\usepackage{physics}
\usepackage{enumerate}
\usepackage{enumitem}

\theoremstyle{plain}
\newtheorem{Th}{Theorem}[section]
\newtheorem{Lemma}[Th]{Lemma}

\newtheorem{Cor}[Th]{Corollary}
\newtheorem{Prop}[Th]{Proposition}

\theoremstyle{definition}
\newtheorem{Def}[Th]{Definition}

\newtheorem{Rem}[Th]{Remark}
\newtheorem{?}[Th]{Problem}

\newcommand{\K}{\mathbb{K}}

\newcommand{\Z}{\mathbb{Z}}

\DeclareMathOperator{\id}{id}
\DeclareMathOperator{\Hom}{Hom}
\DeclareMathOperator{\dep}{dp}
\DeclareMathOperator{\pr}{pr}
\DeclareMathOperator{\spr}{spr}
\DeclareMathOperator{\ad}{ad}

\DeclareMathOperator{\Li}{Li}
\DeclareMathOperator{\Sh}{Sh}

\DeclareSymbolFont{rsfs}{U}{rsfs}{m}{n}
\DeclareSymbolFontAlphabet{\mathscrsfs}{rsfs}

\newcommand{\dk}{\mathfrak{dk}}
\newcommand{\brun}{\mathfrak{brun}}
\newcommand{\tp}{\widetilde{\mathrm{pent}}}
\newcommand{\tc}{\widetilde{\mathrm{corr}}}

\newcommand{\pt}{\mathrm{pent}}

\begin{document}
\title{Brunnian braids and the inclusion from double shuffle Lie algebra to Kashiwara-Vergen Lie algebra.}
\author{Muze Ren}
\address{Institut de Recherche Math\'ematique Avanc\'ee, UMR 7501, Universit\'e de Strasbourg et CNRS, 7
rue Ren\'e Descartes, 67000 Strasbourg, France}
\email{renm@unistra.fr}
\maketitle

\begin{abstract}
Schneps \cite{Schneps2012,Schneps2025} and Enriquez-Furusho \cite{EF4} proved that the double shuffle Lie algebra
$\mathfrak{dmr}_0$ embeds into the Kashiwara--Vergne Lie algebra
$\mathfrak{krv}_2$. We give a Brunnian braid interpretation of a related embedding into the
symmetric Kashiwara--Vergne Lie algebra
$\mathfrak{krv}_2^{\mathrm{sym}}$.
More precisely, the map
\[
\varphi\longmapsto
\bigl(\varphi(-x_0-x_1,x_0),\varphi(-x_0-x_1,x_1)\bigr)
\]
defines an injective Lie algebra homomorphism from the subalgebra
of $\mathfrak{dmr}_0$ satisfying the condition
\[
[x_0,\varphi(-x_0-x_1,x_0)]
+[x_1,\varphi(-x_0-x_1,x_1)]=0
\]
into $\mathfrak{krv}_2^{\mathrm{sym}}$.
The proof reformulate the double shuffle and symmetric
Kashiwara--Vergne relations through abelianizations of Brunnian
Lie algebras associated with the disk and punctured disks.

We generalize this inclusion in two directions.
First, replacing these abelianizations by higher lower central
series quotients yields generalizations of relations and implications
among them. Second, we establish explicit identities relating
the linear pentagon defect to the stuffle coproduct, the
divergence map, and the necklace cobracket.
\end{abstract}

\tableofcontents

\section{Introduction and main results}

For a Riemann surface $M$, the Brunnian braid group
$\mathrm{Brun}_n(M)$ consists of the pure $n$ strand braids that become
trivial after deleting any strand. These braids were studied by
Levinson under the name of decomposable braids \cite{Levinson}.
Brunnian braid groups on the disk and the sphere are related to the
higher homotopy groups of $S^2$ \cite{homotopy}. They also have striking
connections with surface dynamics: for example, every nonidentity
element of the Brunnian subgroup of the mapping class group of a sphere
with at least five punctures is pseudo-Anosov \cite{Ansov}.
In this paper, we use Brunnian braids and their associated graded Lie
algebras to study the relationship between the double shuffle and
Kashiwara--Vergne Lie algebras. The central observation is that their
defining relations can be expressed as same pentagon defect, through different abelianizations corresponding to the Brunnian Lie algebras on different genus zero surfaces.

We begin with the geometric construction underlying this observation.
Let $D$ be a disk and let $\mathrm{PaB}$ denote the category of
parenthesized braids \cite{PaBPaCD}. Consider a morphism
\[
 \Phi\in\Hom_{\mathrm{PaB}(3)}\bigl((12)3,1(23)\bigr)
\]
whose underlying pure braid belongs to $\mathrm{Brun}_3(D)$, this is the standard normalization condition on $\Phi$.
Its pentagon defect is the braid
\[
 \mathrm{Pent}(\Phi)=
 \Phi^{2,3,4}\Phi^{1,23,4}\Phi^{1,2,3}
 (\Phi^{12,3,4})^{-1}(\Phi^{1,2,34})^{-1},
\]
where the superscripts denote the cablings. With composition read from right to left, this is a loop at
$1(2(34))$ in $\mathrm{PaB}(4)$. Deleting any strand makes the two
paths around the pentagon agree, so its underlying braid belongs to
$\mathrm{Brun}_4(D)$.

Let $\K$ be a field of characteristic zero. For a group or Lie algebra,
write $\Gamma_\bullet$ for its lower central series, starting with
$\Gamma_1$ equal to the whole group or Lie algebra.
The relative filtration on $\mathrm{Brun}_n(D)$ is
\begin{equation}\label{eq:relative}
 W_d\mathrm{Brun}_n(D)
 :=\mathrm{Brun}_n(D)\cap\Gamma_d\mathrm{PB}_n(D) 
\end{equation}
and its associated graded Lie algebra over $\K$ is denoted by
$\brun_n(D)$. More generally, $\mathrm{PB}_{m,n}(D)$ is the pure
braid group of $n$ moving strands in a disk with $m$ fixed punctures,
and $\mathrm{Brun}_{m,n}(D)$ is the subgroup killed by deleting any
moving strand. Its relative associated graded Lie algebra is denoted
by $\brun_{m,n}(D)$.
In the four-strand case, these constructions give inclusions
\[
 B:=\brun_4(D)\ \subset\
 I:=\brun_{1,3}(D)\ \subset\
 J:=\brun_{2,2}(D)
 \ \subset\ \mathfrak t_4,
\]
where $\mathfrak t_n$ is the infinitesimal pure braid Lie algebra.
When working with Lie series, we use degree completions of these Lie
algebras and of their lower central series. We write
$L=\mathbb L(x_0,x_1)$ for the completed free Lie algebra and
$L^{\ge d}$ for its subspace of series of degree at least $d$.

For $\varphi\in L^{\ge3}$, put $g=\varphi(t_{12},t_{23})$.
The linear pentagon defect is
\begin{equation}\label{eq:intro_pent}
 \pt(\varphi):=
 g^{1,2,3}+g^{1,23,4}+g^{2,3,4}
 -g^{12,3,4}-g^{1,2,34}\in B.
\end{equation}
This linear map is the associated graded
counterpart of the braid construction above. The equation $\pt(\varphi)=0$ characterizes
the Grothendieck--Teichm\"uller Lie algebra $\mathfrak{grt}_1$, which is proved in \cite{Furusho2011}. The Lie algebra $\mathfrak{grt}_1$ was introduced by V.~Drinfeld \cite{Drinfeld1991} and plays a central role in different branches of math.

The double shuffle Lie algebra $\mathfrak{dmr}_0$ arises from the algebraic study of 
relations among multiple zeta values which is introduced by \cite{Racinet2002}, and the
Kashiwara--Vergne Lie algebra $\mathfrak{krv}_2$ \cite{Alekseev2012} arises from the
Kashiwara--Vergne problem in Lie theory. Furusho proved that
$\mathfrak{grt}_1$ embeds into $\mathfrak{dmr}_0$
\cite{Furusho2011}, and Alekseev and Torossian constructed an
embedding of $\mathfrak{grt}_1$ into $\mathfrak{krv}_2$
\cite{Alekseev2012}. Schneps's proof of an embedding
$\mathfrak{dmr}_0\hookrightarrow\mathfrak{krv}_2$
\cite{Schneps2012} used mould calculus and a special-derivation
property announced by \'Ecalle. Unconditional proofs were
subsequently obtained by Schneps \cite{Schneps2025} and by Enriquez
and Furusho \cite{EF4}.

Our first result gives a Brunnian interpretation of the relations
involved in these inclusions. For $k\ge1$, compose
\eqref{eq:intro_pent} with the natural projection maps to the quotient to obtain
\[
 \begin{aligned}
  \pt^k &:L^{\ge 3}\longrightarrow B/\Gamma_{k+1}B,\\
  \pt^k_{dmr} &:L^{\ge3}\longrightarrow I/\Gamma_{k+1}I,\\
  \pt^k_{krv} &:L^{\ge3}\longrightarrow J/\Gamma_{k+1}J.
 \end{aligned}
\]
Thus $k=1$ corresponds to abelianization. We call
the kernels of those maps the restricted
Grothendieck--Teichm\"uller relations, the generalized double
shuffle relations, and the generalized symmetric Kashiwara--Vergne
relations, respectively. These kernels are vector spaces, and the
inclusions $B\subset I\subset J$ yield the following implications.

\begin{Th}[Theorem~\ref{th:relations}]\label{th:intro-relations}
For $\varphi\in L^{\ge3}$ and $k\ge1$,
\[
 \pt(\varphi)=0
 \ \Longrightarrow\ \pt^k(\varphi)=0
 \ \Longrightarrow\ \pt^k_{dmr}(\varphi)=0
 \ \Longrightarrow\ \pt^k_{krv}(\varphi)=0.
\]
For each of the three families, the relation at level $k$ implies
the relation at level $k-1$ whenever $k\ge2$.
\end{Th}

To state the identification at the first level, set
$x_\infty=-x_0-x_1$ and define the map
\[
 \mathrm{sd}(\varphi)
 :=\bigl(\varphi(x_\infty,x_0),\varphi(x_\infty,x_1)\bigr).
\]
Here a pair $(a,b)$ acts by $x_0\mapsto[x_0,a]$ and
$x_1\mapsto[x_1,b]$. Such a derivation is special precisely when it
annihilates $x_0+x_1$. Accordingly, put
\[
 L^{\mathrm{sym}}:=
 \left\{\varphi\in L^{\ge3}:
 [x_0,\varphi(x_\infty,x_0)]
 +[x_1,\varphi(x_\infty,x_1)]=0\right\}.
\]
This space is the same as $\mathrm{Vep}$ studied in \cite{howarth}. The pair $\mathrm{sd}(\varphi)$ is invariant under exchanging
$x_0$ and $x_1$ together with its two components. We denote the
corresponding fixed subalgebra of $\mathfrak{krv}_2$ by
$\mathfrak{krv}_2^{\mathrm{sym}}$.

\begin{Th}[Theorems~\ref{th:dmr} and~\ref{th:krv}]\label{th:intro_3}
Let $\varphi\in L^{\ge3}$. Then
\begin{enumerate}
 \item $\varphi\in\mathfrak{dmr}_0$ if and only if
       $\pt^1_{dmr}(\varphi)=0$.
 \item $\mathrm{sd}(\varphi)\in\mathfrak{krv}_2^{\mathrm{sym}}$
       if and only if $\varphi\in L^{\mathrm{sym}}$ and
       $\pt^1_{krv}(\varphi)=0$.
\end{enumerate}
therefore, $\mathrm{sd}$ restricts to an injective Lie algebra
homomorphism
\[
 \mathrm{sd}:\mathfrak{dmr}_0\cap L^{\mathrm{sym}}
 \lhook\joinrel\longrightarrow\mathfrak{krv}_2^{\mathrm{sym}}.
\]
\end{Th}

The implication from the double shuffle relation to the symmetric
Kashiwara--Vergne relation is therefore induced by $I\subset J$,
with specialness imposed separately. The double shuffle
characterization is closely related to the polylogarithmic
descriptions in \cite{kernel,howarth} and a different geometric characterizations in \cite{Nikita}; the difference with them is that here we uses the Drinfeld pentagon (Ihara-pentagaon equation for them)
defect directly and imposes no skew-symmetry assumption. The Kashiwara--Vergne characterization reformulates
Kuno's result \cite{Kuno2025} in terms of Brunnian Lie algebras;
see also \cite{GtKontevich_integral} for related emergent braid
constructions.

Our second set of results relates the pentagon defect itself to
operations defining these Lie algebras. On the double shuffle side,
the relevant operation is the stuffle coproduct $\Delta_*$.
Put $R=\K\langle\langle Y\rangle\rangle$, where
$Y=\{y_1,y_2,\ldots\}$ and $y_n$ has weight $n$, and let
$\pi_Y:L\to R$ be the standard projection used in the double
shuffle relations. In the realization of $I$ used in Section~4,
set $\mathcal F_2=\mathbb L(t_{12},t_{23},t_{24})$ and define
\[
 \begin{aligned}
 \mathrm{diag}^{123,124}(\varphi)
   &:=\varphi(t_{12},t_{23})+\varphi(t_{12},t_{24})
      \pmod{[I,I]},\\
 (\varphi)_I^{\mathrm{ab}}
   &:=\overline{\varphi(t_{23},-t_{23}-t_{24})}
      \in I/[I,I],\\
 D_I(\varphi)
   &:=\mathrm{diag}^{123,124}(\varphi)
        -\pt^1_{dmr}(\varphi)+(\varphi)_I^{\mathrm{ab}}.
 \end{aligned}
\]
We regard $I/[I,I]$ as a subspace of $\mathcal F_2/[I,I]$.
Section~4 constructs a monomial coordinate space $\mathcal P$,
an isomorphism $p:R\otimes R\to\mathcal P$, and linear
maps $\bar\pi:\mathcal F_2/[I,I]\to\mathcal P$ and
$\lambda:\mathcal P\to\mathcal P$.
In these coordinates, the full stuffle coproduct has the following
description.\footnote{A similar relation between the pentagon
defect and the stuffle coproduct also appears in an unpublished
note of B.~Enriquez and H.~Furusho.}

\begin{Th}[Theorem~\ref{th:shuffle coproduct}]\label{th:intro-stuffle}
The following diagram of continuous linear maps commutes:
\begin{equation}\label{eq:intro-stuffle-diagram}
\begin{tikzcd}[column sep=huge,row sep=large]
 L^{\ge2}
   \arrow[r,"\pi_Y"] \arrow[d,"D_I"']
 & R \arrow[dd,"\Delta_*"] \\
 \mathcal F_2/[I,I]
   \arrow[d,"\lambda\circ\bar\pi"'] & \\
 \mathcal P \arrow[r,"p^{-1}"'] & R\otimes R.
\end{tikzcd}
\end{equation}
\end{Th}

On the Kashiwara--Vergne side, the corresponding operations are
divergence and the necklace cobracket. Let
$A=\K\langle\langle x_0,x_1\rangle\rangle$ and let
$C=A/[A,A]$ be the space of cyclic words.
Write $|-|:A\to C$ for the quotient map. The coproduct $\Delta$
for which $x_0,x_1$ are primitive and its antipode $S$ induce
\[
 \widetilde\Delta(|a|)
   =(|-|\otimes|-|)(\id\otimes S)\Delta(a).
\]
We use the full algebraic necklace cobracket
$\delta_{\mathrm{necklace}}:C\to C\otimes C$;
projecting both output factors modulo constants gives the reduced
necklace cobracket. Section~\ref{sec:krv} constructs the corrected
lift $\tp_{\mathrm{sym}}:L^{\mathrm{sym}}\to A$ of the
abelianized pentagon defect; see \eqref{eq:s5-lifts}.

\begin{Th}[Theorem~\ref{th:krv_commutatvie}]\label{th:intro-divergence}
For each $d\ge3$, the following diagram of linear maps commutes:
\[
\begin{tikzcd}[column sep=large,row sep=large]
 L_d^{\mathrm{sym}}
   \arrow[r,"\mathrm{sd}"]
   \arrow[d,"\tp_{\mathrm{sym}}"']
 & (\mathfrak{sder}_2^{\mathrm{sym}})_d
   \arrow[r,"h"] \arrow[d,"\operatorname{div}"]
 & C_{d+1} \arrow[d,"-\delta_{\mathrm{necklace}}"] \\
 A_d \arrow[r,"\frac1d|-|"']
 & C_d \arrow[r,"\widetilde\Delta"']
 & (C\otimes C)_d.
\end{tikzcd}
\]
Here $d$ is the component degree of a tangential derivation, and
\[
 h(a,b)=-\frac{1}{d+1}|x_0a+x_1b|.
\]
In particular,
\[
 \operatorname{div}\bigl(\mathrm{sd}(\varphi)\bigr)
   =\frac1d\bigl|\tp_{\mathrm{sym}}(\varphi)\bigr|
 \qquad(\varphi\in L_d^{\mathrm{sym}}).
\]
\end{Th}

Finally, we want to address the question of using the finite type invariants and low dimensional topology to study the spaces $\pt^k,\pt^k_{dmr},\pt^k_{krv}$. 

As a initial step towards this question and motivated by the facts that Milnor invariants \cite{ML}  detects the Brunnian braids, we investigate how the pentagon braid
can be detected by finite type invariants.
Two filtrations enter this question: the relative filtration \eqref{eq:relative}
inherited from the ambient pure braid group, and the intrinsic
lower central series of the Brunnian group. 

Choose $m\ge0$ and $n\ge1$ with $m+n=4$, and write
\[
 P=\mathrm{PB}_{m,n}(D),\qquad
 H=\mathrm{Brun}_{m,n}(D),\qquad
 F=\ker\bigl(P\to\mathrm{PB}_{m,n-1}(D)\bigr).
\]
Here the map deletes the last moving strand.
The relative filtration satisfies
\[
 W_dH=H\cap\Gamma_dF=H\cap\Gamma_dP.
\]
For $b\in H\setminus\{1\}$, define
\[
 \operatorname{ord}_W(b)=\max\{d:b\in W_dH\},\qquad
 \operatorname{ord}_H(b)=\max\{r:b\in\Gamma_rH\}.
\]
Let $\mathcal M$ be the Magnus expansion of $F$, and let
$\mathcal Z$ be a group-like Kontsevich expansion of $P$ with normalization.
For the conjugation action of $P$ on
$K_0=F$ and $K_r=\Gamma_rH$ for $r\ge1$, put
\[
 \mathcal J_r(P)=
 \{g\in P:[g,K_j]\subset K_{r+j}\text{ for every }j\ge0\},
\]
and let $\tau_r$ be the corresponding Johnson homomorphism which is constructed using the theory of extended Johnson--Morita theory that is developed in \cite{JM}.

\begin{Th}[Corollary~\ref{cor:nt-final-pentagon}]\label{th:intro-orders}
Let $b=\mathrm{Pent}(\Phi)$ for a normalized $\Phi$ as above.
Regard its underlying braid as an element of $H$. If $b\ne1$, then
\begin{align*}
 \operatorname{ord}_W(b)=d
 &\quad\Longleftrightarrow\quad
 \mathcal M(b)-1\text{ begins in degree }d\\
 &\quad\Longleftrightarrow\quad
 \mathcal Z(b)-1\text{ begins in degree }d,\\[2pt]
 \operatorname{ord}_H(b)=r
 &\quad\Longleftrightarrow\quad
 b\in\mathcal J_r(P)\text{ and }\tau_r(b)\ne0.
\end{align*}
\end{Th}

Therefore the Magnus and Kontsevich expansions detect relative order,
whereas Johnson homomorphisms detect intrinsic Brunnian order.

The paper is organized as follows. Section~2 develops mixed braid
groups and their Brunnian filtrations. Section~3 introduces the
pentagon defects and proves the implications among their quotient
relations. Section~4 treats the double shuffle relations and the
stuffle coproduct. Section~5 treats the symmetric
Kashiwara--Vergne relations, divergence, and the necklace cobracket.
Section~6 study the pentagon braid.

\medskip
\noindent\textbf{Acknowledgements.}
The author is grateful to A.~Alekseev and R.~Kashaev for their
encouragement, and to B.~Enriquez for suggestions on the mathematics
and exposition and for encouragement to write up these results.
The author also thanks D.~Bar-Natan, H.~Furusho, G.~Massuyeau,
F.~Naef, and G.~Laplante-Anfossi for answering questions about their
papers and for helpful suggestions. Part of this work was presented
at the conference \emph{Homological, Quantum, and Computational
Methods in Low-Dimensional Topology}; the author thanks the
organizers for the invitation. The author is supported by the SNSF
Postdoc.Mobility grant P500PT\_230340.

\section{Mixed braid group and Brunnian filtrations}

Let $M$ be a Riemann surface (mainly disk $D$ or sphere $S^2$ in this note), the configuration space of $n$-th ordered points of $M$ is
\begin{equation*}
\mathrm{Conf}(M,n)=\{(x_1,\ldots,x_{n})\in M\times \ldots M\mid x_i\ne x_j,~ \mathrm{for}~i\ne j\}. 
\end{equation*} For $1\le i\le n$, the coordinate projection map
\begin{equation*}
    p[M]_n^i: \mathrm{Conf}(M,n)\to \mathrm{Conf}(M,n-1),\quad  p[M]_n^i(x_1,\ldots x_{n})=(x_1,\ldots,x_{i-1},x_{i+1},\ldots ,x_{n}),
\end{equation*}
induces the group homomorphism 
\begin{equation*}
    \pi_1(p[M]_n^i): \pi_1(\mathrm{Conf(M,n}))\to \pi_1(\mathrm{Conf(M,n-1})).
\end{equation*}

In this note, we are interesed in the following groups related to the $\pi_1(\mathrm{Conf}(M,n))$.

\begin{enumerate}
    \item $\mathrm{PB}_n(M)$, the n strand pure braid group over $M$, it is the same  as $\pi_1(\mathrm{Conf}(M,n))$, the map $\pi_1(p[M]_n^i)$ is the same as deleting the $i$ th strand of the pure braid group.\\
    \item  $\mathrm{Brun}_n(M)$ is the n strand Brunnian braids group on $M$, \[\mathrm{Brun}_n(M):= \ker \pi_1(p[M]_n^1)\cap \ker \pi_1(p[M]_n^2)\cap\dots \cap  \ker \pi_1(p[M]_n^n),\] those are the braids that become trivial after deleting any of the strand.\\
    \item $\mathrm{PB}_{m,n}(M)$ is the $n+m$ strand mixed pure braid group with $m$ fixed strands, it is defined as 
    \begin{equation*}
       \mathrm{PB}_{m,n}(M):=\ker \left( \pi_1(p[M]^{m+1}_{m+1})\circ \pi_1(p[M]^{m+2}_{m+2})\circ \ldots \pi_1(p[M]^{n+m}_{n+m})\right).
    \end{equation*}
    For $M=D$, this is also the fundamental group of configuration space of n points on a $m$ punctured $D$.\\
    \item $\mathrm{Brun}_{m,n}(M)$ is the $n+m$ strand mixed brunnian braids with $m$ fixed strand,
    \begin{equation*}
    \ker \pi_1(p[M]_{n+m}^{m+1})\cap \ker \pi_1(p[M]_{m+n}^{m+2})\cap\dots \cap  \ker \pi_1(p[M]^{m+n}_{m+n})).  
    \end{equation*}
\end{enumerate}

The relations of the groups are summarized into the following diagram
\begin{equation*}
    \begin{tikzcd}
        \mathrm{Brun}_{1,n}(M)\arrow[d,hook]\arrow[r,hook]& \mathrm{Brun}_{2,n-1}(M)\arrow[r,hook]\arrow[d,hook]& \mathrm{Brun}_{3,n-2}(M)\ldots\arrow[r,hook]\arrow[d,hook]&\mathrm{Brun}_{n-1,2}(M)\arrow[r,hook]\arrow[d,hook]&\mathrm{Brun}_{n,1}(M)\arrow[d,equal]\\
         \mathrm{PB}_{1,n}(M)\arrow[d,hook]& \mathrm{PB}_{2,n-1}(M)\arrow[l,hook]& \mathrm{PB}_{3,n-2}(M)\ldots\arrow[l,hook]&\mathrm{PB}_{n-1,2}(M)\arrow[l,hook]&\mathrm{PB}_{n,1}(M)\arrow[l,hook]\\
          \mathrm{PB}_{n+1}(M)& & &&
    \end{tikzcd}
\end{equation*}

\subsection{Associated graded Lie algebras}
Throughout, \(\K\) is a field of characteristic zero.

Recall that for a group $G$, the descending lower central series is the filtration
\begin{equation*}
    G=\Gamma_1(G)\ge \Gamma_2(G)\ge \ldots \Gamma_n(G)\ge \ldots,
\end{equation*}
where $\Gamma_1(G)=G$, $\Gamma_{i+1}(G):=[G,\Gamma_i(G)],$ the associated graded over the field $\K$ is a graded Lie algebra $\mathrm{gr}_{\K}(G):=\left(\oplus_i\Gamma_i(G)/\Gamma_{i+1}(G)\right)\otimes \K$. For a subgroup $H$ of $G$, the relative lower central series with respect to $G$ is
\begin{equation*}
    H\cap \Gamma_1(G)\ge H\cap \Gamma_2(G)\ge \ldots H\cap \Gamma_n(G)\ge \ldots,
\end{equation*}
the associated graded with respect to the relative filtration is $\mathrm{gr}^R_{\K}(H):=(\oplus_i \Gamma_i(G)\cap H/\Gamma_{i+1}(G)\cap H)\otimes_{\Z} \K.$

The associated graded Lie algebra of the above group and the relative associated graded Lie algebra of the above subgroups are the following.

\begin{enumerate}
    \item $\mathrm{gr}_{\K}(\mathrm{PB}_n(D))$, the associated graded Lie algebra of the disk pure braid group is the Drinfeld-Kohno Lie algebra \cite{Drinfeld1991}. The Lie algebra $\mathfrak{t}_n$ is generated by $t_{ij},1\le i,j\le n$, and $t_{ij}=t_{ji},t_{ii}=0$, with the following relations:
	\begin{align*}
		&[t_{ij},t_{kl}]=0,\quad \text{for $i,j,k,l$ distinct}\\
		&[t_{ij},t_{ik}+t_{jk}]=0, \quad \text{for $i,j,k$ distinct}.
	\end{align*}
    The group morphism $\pi_1[p(D)_n^i]$ induces the Lie algebra homomorphism,
    \begin{equation*}
        \pr^i_n:\mathfrak{t}_n\to \mathfrak{t}_{n-1}; t_{il}\mapsto 0, t_{kl}\mapsto t_{kl}, i\ne k,l.
    \end{equation*}
    \item $\mathrm{gr}_{\K}(\mathrm{PB}_n(S^2))$, the associated graded Lie algebra of the sphere pure braid group is the spherical Drinfeld-Kohno Lie algebra \cite{Ihara1992}. $\mathfrak{p}_{n}$ is generated by $X_{ij} =X_{ji}$, for $1 \leq i,j \leq n$, subject to the relations
    \begin{equation*}
        \begin{split}
            & X_{ii} =0, \quad \forall i \in \{1, \ldots, n\};  \sum_{j=1}^{n} X_{ij} = 0, \quad \forall i \in \{1, \ldots, n\}; \\
            & [X_{ij},X_{kl}] = 0, \quad \text{if } \{i,j\} \cap \{k,l\} = \emptyset.
        \end{split}
    \end{equation*}
    The group morphism $\pi_1[p(S^2)_n^i]$ induces the Lie algebra homomorphism,
    \begin{equation*}
        \spr^i_n:\mathfrak{p}_n\to \mathfrak{p}_{n-1}; X_{il}\mapsto 0, X_{kl}\mapsto X_{kl}, \forall  k,l\ne i.
    \end{equation*}
    \item $\mathrm{gr}_{\K}(\mathrm{PB}_{m,n}(D))$, the associated graded Lie algebra of the mixed braid group on disk, \begin{equation*}
    \dk_{m,n}(D):=\ker(\rm{pr}^{m+1}_{m+1}\circ \rm{pr}^{m+2}_{m+2}\dots \rm{pr}^{m+n}_{m+n}). 
     \end{equation*}
It is generated by the generators $t_{ij}$ with either $m< j\le m+n$ or $m< i\le m+n$. $\dk_{m,n}(D)$ is  introduced and studied in \cite{Kuno2025}, see also \cite{mixed_braid}. 
     \item Relative to the inclusion from $\mathrm{Brun}_n(D)$ to $\mathrm{PB}_n(D)$, $\mathrm{gr}^R_{\K}(\mathrm{Brun}_n(D))$ is the Lie subalgebra of $\mathfrak{t}_n$. It is identified with
     \begin{equation*} \brun_n(D):=\cap^n_{i=1}\ker \pr_n^i,
     \end{equation*}
     which is proved in \cite{Brun_Lie}.
     
\item  Relative to the inclusin from $\mathrm{Brun}_{m,n}(D)$ to $\mathrm{PB}_{m,n}(D)$, $\mathrm{gr}^R_{\K}(\mathrm{Brun}_{m,n}(D))$ is the Lie subalgebra of $\dk_{m,n}(D)$ and is identified with $\brun_{m,n}(D):=\cap^{m+n}_{i=m+1}\ker \pr^{i}_{m+n}$ by the following Proposition \ref{prop:mixed-brunnian-relative-graded}.
\end{enumerate}

The following Proposition is implicitly proved in \cite{WJ} when discussing bi-$\Delta$ groups, we included the proofs for completeness. 
\begin{Prop}[\cite{WJ}]\label{prop:mixed-brunnian-relative-graded}
Let $m\geq0$, $n\geq1$, and $N=m+n$. Put
$G=\mathrm{PB}_{m,n}(D)$ and $H=\mathrm{Brun}_{m,n}(D)$, and
equip $H$ with the relative filtration
\[
 F_rH=H\cap\Gamma_r(G).
\]
Under the identification
$\mathrm{gr}_{\K}(G)=\dk_{m,n}(D)$, inclusion identifies
\[
 \mathrm{gr}^{R}_{\K}(H)
 :=\bigoplus_{r\geq1}(F_rH/F_{r+1}H)\otimes_{\mathbb Z}\K
 \quad\text{with}\quad
 \brun_{m,n}(D)
 :=\bigcap_{i=m+1}^{N}\ker\pr_N^i.
\]
And this is a graded Lie ideal of $\dk_{m,n}(D)$.
\end{Prop}
\begin{proof}
For convenience, in the disc case, we denote
\[
 d_i:\mathrm{PB}_{m,n}(D)\longrightarrow
       \mathrm{PB}_{m,n-1}(D)
 \qquad(m<i\leq N)
\]
for deletion of the $i$-th strand and the insertion
homomorphisms $s_i$ adds a trivial moving strand in the $i$-th
position, in front of the other strands. (This is not the cabiling operation we discussed later). 

They restrict to mixed
pure braid groups because they preserve the first $m$ fixed
strands. They satisfy
\begin{equation}\label{eq:brunnian-deletion-insertion}
 d_i s_i=\mathrm{id},\qquad
 d_j s_i=s_i d_{j-1},\qquad
 d_{j-1}d_i=d_i d_j\quad(i<j).
\end{equation}

First work over $\mathbb Z$. In each degree, inclusion induces
an injection
\[
 F_rH/F_{r+1}H\hookrightarrow
 \Gamma_r(G)/\Gamma_{r+1}(G).
\]
Its image is contained in
$\bigcap_{i=m+1}^{N}\ker(\mathrm{gr}_r d_i)$, since every
deletion $d_i$ is trivial on $H$.

Conversely, let $\xi$ belong to this intersection, and choose
$g\in\Gamma_r(G)$ representing $\xi$. Then
\[
 d_i(g)\in\Gamma_{r+1}(\mathrm{PB}_{m,n-1}(D))
 \qquad(m<i\leq N).
\]
Starting with $i=N$ and proceeding down to $m+1$, replace the
current representative $g$ by
\[
 g'=g\,s_i(d_i(g))^{-1}.
\]
The difference $s_i(d_i(g))$ lies in $\Gamma_{r+1}(G)$, so $g'$ still
represents $\xi$, while $d_i(g')=1$. Furthermore, if
$d_j(g)=1$ for some $j>i$, then
\[
 d_j\bigl(s_i(d_i(g))\bigr)
 =s_i\bigl(d_{j-1}d_i(g)\bigr)
 =s_i\bigl(d_i d_j(g)\bigr)=1.
\]
Thus each correction preserves all previously imposed deletion
conditions. After the last step, the representative belongs to
$H\cap\Gamma_r(G)=F_rH$. This proves equality of the two
degree-$r$ subgroups.

In each degree put
$D_i=\mathrm{gr}_r d_i$ and $S_i=\mathrm{gr}_r s_i$.
The identities \eqref{eq:brunnian-deletion-insertion} show that
\[
 \Pi_r=(\mathrm{id}-S_{m+1}D_{m+1})\cdots
       (\mathrm{id}-S_ND_N)
\]
is a linear projector onto $\bigcap_{i=m+1}^{N}\ker D_i$;
the rightmost factor acts first. Because the factors successively
kill the deletions in descending order and preserve those already
killed, and each factor fixes the common kernel.
The same identities hold after tensoring with $\K$.
So the integral inclusion is split and its scalar extension
has exactly the asserted image. Under
$\mathrm{gr}_{\K}(G)=\dk_{m,n}(D)$, the maps $D_i$ become
the restrictions of $\pr_N^i$.
The intersection computed in $\mathfrak t_N$ already lies in
$\dk_{m,n}(D)=\ker(\mathfrak t_N\to\mathfrak t_m)$, because
forgetting all moving strands factors through each moving-strand
deletion.
Finally, the common kernel is a Lie ideal, being an intersection
of kernels of Lie algebra homomorphisms.
\end{proof}

\begin{Rem}
The same relative filtration on $H$ is induced by
$\mathrm{PB}_{N}(D)$. Let
$F=\ker(d_N:\mathrm{PB}_{N}(D)\to\mathrm{PB}_{N-1}(D))$.
The standard almost-direct decomposition
$\mathrm{PB}_{N}(D)=F\rtimes\mathrm{PB}_{N-1}(D)$ gives
$F\cap\Gamma_r(\mathrm{PB}_{N}(D))=\Gamma_r(F)$.
Since $H\subseteq F\subseteq G\subseteq\mathrm{PB}_{N}(D)$,
\[
 H\cap\Gamma_r(F)
 \subseteq H\cap\Gamma_r(G)
 \subseteq H\cap\Gamma_r(\mathrm{PB}_{N}(D))
 =H\cap\Gamma_r(F).
\]
Thus both ambient groups give the same filtration.
\end{Rem}

\subsection{Related extended $N$ series}\label{subsec:extended N series}
Recall that an \emph{N-series} of a group $K$ is a descending series
\begin{equation*}
    K=K_1\ge K_2\ge \ldots \ge K_i\ge \ldots,
\end{equation*}
such that $[K_m,K_n]\le K_{m+n},$ for $m,n \ge 1$.  An \emph{extended} $N$ series $K_{*}=(K_m)_{m\ge 0}$ is a descending series
\begin{equation*}
    K_0\ge K_1\ge \ldots \ge K_k\ge \ldots
\end{equation*}
such that $[K_m,K_n]\le K_{m+n}$ for $m,n\ge 0$. For every extended N-series $(K_{*})=(K_m)_{m\ge 0}$, the subseries $K_{+}=(K_m)_{m\ge 1}$ is an $N$ series, this is introduced in \cite{JM}. We are mainly interested in the following two extended $N$ series that are related to the Johnson--Morita theory constructed in Section 6.

\begin{Def}
The kernel of the projection $\pi_1(p[D]^{n+m}_{n+m})$ is the free group of rank $n+m-1$.
    The Brunnian filtration of the free group $F_{m+n-1}$ is the extended N-series
    \begin{align*}
        F_{m+n-1}\ge \mathrm{Brun}_{m,n}(D)\ge \Gamma_2(\mathrm{Brun}_{m,n}(D))\ge \ldots \Gamma_i(\mathrm{Brun}_{m,n}(D))\ge \ldots.
    \end{align*}

    The relative Brunnian filtration of the group $\mathrm{PB}_{m,n}(D)$ is the extended N-series
    \begin{align*}
        \mathrm{PB}_{m,n}(D)\ge \mathrm{Brun}_{m,n}(D)\ge \left(\Gamma_2(\mathrm{PB}_{m,n}(D))\cap \mathrm{Brun}_{m,n}(D)\right)\ge \ldots \left(\Gamma_{k}(\mathrm{PB}_{m,n}(D))\cap \mathrm{Brun}_{m,n}(D)\right)\ge \ldots.
    \end{align*}
\end{Def}

There exists the corresponding notion of extended N-series in the Lie algebra setting, let $\mathbb{L}(S)$ denote the free Lie algebra over some alphabet and $\mathbb{L}_{n+m}$ the rank $n+m$ free Lie algebra.

\begin{Def}
    The Brunnian filtration of the free Lie algebra is the following extended N-sereis
    \begin{equation*}
        \mathbb{L}_{n+m-1}\ge \brun_{m,n}(D)\ge [\brun_{m,n}(D),\brun_{m,n}(D)]\ge \Gamma_3(\brun_{m,n}(D))\ge \ldots \Gamma_n(\brun_{m,n}(D))\ge \ldots.
    \end{equation*}

    The relative Brunnian filtration of the Lie algebra $\dk_{m,n}$ is the following extended N-series
    \begin{equation*}
        \dk_{m,n}(D)\ge \brun_{m,n}(D)\ge (\Gamma_2(\dk_{m,n}(D))\cap \brun_{m,n}(D))\ge \ldots (\Gamma_k(\dk_{m,n}(D))\cap \brun_{m,n}(D))\ge \ldots.
    \end{equation*}
\end{Def}

\section{Pentagon equation and brunnian filtration}

In this section, we introduce the defect of pentagon equation and study how it interacts with the Brunnian filtration. Unless otherwise stated, in the following, Lie algebras and associative algebras are completed with respect to the weight grading, tensor products are completed tensor products, and ideals and commutator subspaces are closed. 

For a partially defined map $f:\{1,\ldots,m\}\to \{1,\ldots,n\}$, the Lie algebra morphism $\mathfrak{t}_n\to \mathfrak{t}_m:~g\mapsto g^f=g^{f^{-1}(1),\ldots,f^{-1}(n)}$ is uniquely defined by $(t_{ij})^{f}=\sum_{ i'\in f^{-1}(i),j'\in f^{-1}(j)}t_{i',j'}$ for $i\ne j$. 

For any element $g$ in $\mathfrak{t}_3$, the pentagon map $\pt:\mathfrak{t}_3\to \mathfrak{t}_4$ is defined to be
\begin{equation}
\pt(g):=g^{1,2,3}+g^{1,23,4}+g^{2,3,4}-g^{12,3,4}-g^{1,2,34}.
\end{equation}

 For $n=3$, $\mathfrak{t}_3$ is generated by $t_{12},t_{23},t_{13}$ and it has the decomposation into its center and a free Lie algebra generated by $t_{12},t_{13}$,
\begin{equation}
    \mathfrak{t}_3\simeq \K(t_{12}+t_{13}+t_{23})\oplus \mathbb{L}(t_{12},t_{23}),
\end{equation}
it also have another decomposition 
\begin{equation*}
    \mathfrak{t}_3\simeq \mathbb{L}(t_{13},t_{23})\rtimes \K t_{12}.
\end{equation*}

For $\varphi\in \mathbb{L}^{>1}(x_0,x_1)$, we use the notation $\pt\left(\varphi(t_{12},t_{23})\right):=\pt|_{\mathbb{L}(t_{12},t_{23})}(\varphi)$ and similarly $\pt\left(\varphi(t_{13},t_{23})\right):=\pt|_{\mathbb{L}(t_{13},t_{23})}(\varphi)$, in those cases, more explicitly we have
\begin{multline}
\pt(\varphi(t_{12},t_{23}))=\varphi(t_{23},t_{34})-\varphi(t_{13}+t_{23},t_{34})+\\\varphi(t_{12}+t_{13},t_{24}+t_{34})-\varphi(t_{12},t_{23}+t_{24})+\varphi(t_{12},t_{23}), 
\end{multline}
\begin{multline}
\pt(\varphi(t_{13},t_{23}))=\varphi(t_{13},t_{23})+\varphi(t_{14},t_{24}+t_{34})+\varphi(t_{24},t_{34})\\
-\varphi(t_{14}+t_{24},t_{34})-\varphi(t_{13}+t_{14},t_{23}+t_{24})
\end{multline}

\begin{Lemma}
For $\varphi\in \mathbb{L}^{>1}(x_0,x_1)$, the relation $\pt(\varphi(t_{12},t_{23}))=\pt(\varphi(-t_{13}-t_{23},t_{23}))$ holds in $t_{4}$.
\end{Lemma}

\begin{proof}
    For $\varphi$ in $\mathbb{L}^{>1}(x_0,x_1)$, $\varphi$ is a sum of commutators of $x,y$. $(t_{12}+t_{13}+t_{23})$ is in the center, therefore  the relation $\varphi(t_{12},t_{23})=\varphi(t_{12}-(t_{12}+t_{13}+t_{23}),t_{23})=\varphi(-t_{13}-t_{23},t_{23})$ holds.
\end{proof}

\begin{Lemma}\label{lem:projection}
For $\varphi\in \mathbb{L}^{>1}(x_0,x_1)$, we have
\[\pt(\varphi(t_{12},t_{23}))\in \ker \pr^1_4\cap \ker\pr^2_4 \cap \ker \pr^3_4 \cap \ker \pr^4_4.\]  
\end{Lemma}

\begin{proof}
We first apply $\pr^1_4$ to $\pt(\varphi(t_{12},t_{23}))$,
\begin{align*}
\pr^1_4(\pt(\varphi))=\varphi(t_{23},t_{34})-\varphi(t_{23},t_{34})=0,
\end{align*}
for other projections, the calculations are similar.
\end{proof}

From the Lemma \ref{lem:projection}, the map has the property that
\begin{align}
    &\pt(\varphi(t_{12},t_{23}))\in \brun_4(D)\subset \mathfrak{t}_4,\quad \pt(\varphi(t_{12},t_{23}))\in \brun_{1,3}(D)\subset \dk_{1,3}(D)\\
    &\pt(\varphi(t_{12},t_{23}))\in \brun_{2,2}(D)\subset \dk_{2,2}(D).
\end{align}

If we denote the center of $\mathfrak{t}_n$ by $Z(\mathfrak{t}_n)$, it is one dimensional, $Z(\mathfrak{t}_n)=\K(\sum_{1\le i<j\le n}t_{ij})$. The map \begin{equation*}
	K_n: \mathfrak{t}_n\to \mathfrak{p}_{n+1}, t_{ij}\mapsto X_{ij},
\end{equation*}
$\ker K_n=Z(\mathfrak{t}_n)$, therefore it induces an isomorphism $ \mathfrak{t}_n/(Z(\mathfrak{t}_n))\simeq \mathfrak{p}_{n+1}.$

By the relations of $\mathfrak{p}_n$, $\mathfrak{p}_4\simeq \mathbb{L}(x_0,x_1)$, where $x_1=X_{14}=X_{23}, x_0=X_{12}=X_{34}, -x_0-x_1=X_{13}=X_{24}$, the map $\pt: \mathfrak{p}_4\to \mathfrak{p}_5$ is defined to be
\begin{multline}   \pt(\varphi(X_{12},X_{23})):=\varphi(X_{23},X_{34})-\varphi(X_{13}+X_{23},X_{34})+\\
    \varphi(X_{12}+X_{13},X_{24}+X_{34})-\varphi(X_{12},X_{23}+X_{24})+\varphi(X_{12},X_{23})\in \mathfrak{p}_5
\end{multline}

For later use, we compare $\brun_{2,2}(D)$ and $\brun_{1,3}(D)$ with Lie algebras in $\mathfrak{p}_5$.

\begin{Prop}\label{prop:lambda_13}

The map $K_4$ identifies $\ker \mathrm{pr}^4_4$ and $\ker \spr^4_5$, it further identify the $\brun_{2,2}(D)$ with $ \ker \spr^4_5\cap\ker \spr^3_5$ and $\brun_{1,3}(D)$ with $ \ker \spr^2_5\cap \ker \spr^3_5\cap \ker \spr^4_5$, that is
    \begin{align*}
    &K_4: \ker\pr^4_4\cap \ker \pr^3_4\simeq \ker \spr^4_5\cap\ker \spr^3_5;\\ &K_4: \ker\pr^4_4\cap \ker \pr^3_4\cap \ker \pr^2_4\simeq \ker \spr^2_5\cap \ker \spr^3_5\cap \ker\spr^4_5
    \end{align*}
\end{Prop}
\begin{proof}
	It is known that $\ker \mathrm{pr}^4_4$ is the free Lie algebra $\mathbb{L}(t_{14},t_{24},t_{34})$ and $\ker K_4$ is the center $\K(\sum_{1\le i<j\le 4}t_{ij})$, therefore $\ker K_4\cap \ker \pr^4_4=0$, we have $K_4\mid_{\ker \pr^4_4}$ is injective. From the following commutative diagram, the image of $K_4\mid_{\ker \pr^4_4}$ is $\ker \spr^4_5$, for any $w\in \ker \spr_5^4$, we choose a $v$ such that $K_4(v)=w$, then $\pr^4_4(v)=c(\sum_{1\le i<j\le 3}t_{ij})\in \ker K_3$ for some $c\in \K$. We also have $\pr_4^4((\sum_{1\le i<j\le 4}t_{ij})=(\sum_{1\le i<j\le 3}t_{ij})$, therefore $K_4(v-c(\sum_{1\le <j\le 4}t_{ij})=w$, $K_4\mid_{\ker\pr_4^4}$ is an isomorphism which maps $t_{i4}$ to $X_{i4}$, $i=1,2,3$.
     
	\[
	\begin{tikzcd}[ampersand replacement=\&, column sep=large]
		\mathfrak{t}_4 \arrow{r}{K_4}\arrow{d}{\pr^4_4}                                       \&  \mathfrak{p}_5\arrow{d}{\spr^4_5}                       \\
		\mathfrak{t}_3\arrow{r}{K_3}\& \mathfrak{p}_4
	\end{tikzcd}.
	\]
     After the isomorphism between the two free Lie algebra, we have 
     \begin{equation*}
      \ker\pr^4_4\cap \ker \pr^3_4=(t_{34})\simeq \ker\spr^4_5\cap \ker \spr^3_5=(X_{34}).   
     \end{equation*}
     Similarly we have
     \begin{equation*}
     \ker\pr^4_4\cap \ker \pr^3_4\cap \ker \pr^2_4=[(t_{24}),(t_{34})]\simeq \ker \spr^2_5\cap \ker \spr^3_5\cap \ker \spr^4_5=[(X_{24}),(X_{34})].    
     \end{equation*}
\end{proof}

\subsection{Various generalizations of $\mathfrak{grt}_1$}
\begin{Def}[\cite{Drinfeld1991,Furusho2011}]
The \emph{Grothendieck-Teichmuller Lie algebra} $\mathfrak{grt}_1(\K)$ consists of $\varphi\in \mathbb{L}^{\ge 3}(x_0,x_1)$, which satisfies the conditions
\begin{align*}
\varphi(t_{12},t_{23}+t_{24})+\varphi(t_{13}+t_{23},t_{34})=\varphi(t_{23},t_{34})+\varphi(t_{12}+t_{13},t_{24}+t_{34})+\varphi(t_{12},t_{23})
\end{align*}
\end{Def}

For the genus zero disk case, we generalize the $\mathfrak{grt}_1$ relations in three directions.

For the filtration $\brun_4(D)\ge \Gamma_2(\brun_4(D))\ge \ldots \Gamma_{n}(\brun_4(D))\ldots$, we introduce the map
    \begin{equation}
        \pt^k:\mathbb{L}(x_0,x_1)^{\ge 3}\to \brun_4(D)/\Gamma_{k+1}(\brun_4(D)),\quad \pt^k=\mathrm{proj}\circ \pt(\varphi(t_{12},t_{23}),
    \end{equation}
    where $\mathrm{proj}$ is the projection from $\brun_4(D)$ to $\brun_4(D)/\Gamma_{k+1}(\brun_4(D)).$

\begin{Def}[Restricted Grothendieck-Teichmuller relations] $\varphi\in \mathbb{L}^{\ge 3}(x_0,x_1)$ satisfies $k$-th restricted Grothendieck-Teichmuller relations if
\begin{align}
    \pt^k(\varphi)=0.
\end{align}    
\end{Def}

For the filtrations $\brun_{1,3}(D)\ge \Gamma_2(\brun_{1,3}(D))\ge \ldots \Gamma_n(\brun_{1,3}(D))\ldots$, we introduce the map
    \begin{equation}
        \pt_{dmr}^k:\mathbb{L}(x_0,x_1)^{\ge 3}\to \brun_{1,3}(D)/\Gamma_{k+1}(\brun_{1,3}(D)),\quad \pt^k_{dmr}=\mathrm{proj}\circ \pt(\varphi(t_{12},t_{23}),
    \end{equation}
    where $\mathrm{proj}$ is the projection from $\brun_{1,3}(D)$ to $\brun_{1,3}(D)/\Gamma_{k+1}(\brun_{1,3}(D)).$

\begin{Def}[Generalized double shuffle relations] $\varphi\in \mathbb{L}^{\ge 3}(x_0,x_1)$ satisfies $k$-th generalized double shuffle relations if
\begin{align}
    \pt_{dmr}^k(\varphi)=0.
\end{align}    
\end{Def}

For the filtrations $\brun_{2,2}(D)\ge \Gamma_2(\brun_{2,2}(D))\ge \ldots \Gamma_k(\brun_{2,2}(D))$, we introduce the map
    \begin{equation}
        \pt_{krv}^k:\mathbb{L}(x_0,x_1)^{\ge 3}\to \brun_{2,2}/\Gamma_{k+1}(\brun_{2,2}),\quad \pt^k_{krv}=\mathrm{proj}\circ \pt(\varphi(t_{12},t_{23}),
    \end{equation}
    where $\mathrm{proj}$ is the projection from $\brun_{2,2}$ to $\brun_{2,2}(D)/\Gamma_{k+1}(\brun_{2,2}(D)).$

\begin{Def}[Generalized symmetric Kashiwara-Vergne relations] $\varphi\in \mathbb{L}^{\ge 3}(x_0,x_1)$ satisfies $k$-th symmetric Kashiwara-Vergne relations if
\begin{align}
    \pt_{krv}^k(\varphi)=0.
\end{align}    
\end{Def}

Our main result in this section is the relations between those vector spaces. 
\begin{Th}\label{th:relations}
    For $\varphi\in \mathbb{L}^{\ge 3}(x,y)$, we have the relations

    \begin{enumerate}       
        \item $\pt(\varphi)=0\Rightarrow \pt^k(\varphi)=0\Rightarrow \pt^{k-1}(\varphi)=0$ for $k\ge 2$.
        \item $\pt^k_{dmr}(\varphi)=0\Rightarrow  \pt^{k-1}_{dmr}(\varphi)=0$, for $k\ge 2$.
        \item $\pt^k_{krv}(\varphi)=0\Rightarrow  \pt^{k-1}_{krv}(\varphi)=0$, for $k\ge 2$.
        \item $\pt^k(\varphi)=0\Rightarrow \pt^k_{dmr}(\varphi)=0 \Rightarrow \pt^k_{krv}(\varphi)=0,$ for $k\ge 1$.
    \end{enumerate}
\end{Th}

\begin{proof}
    It follows from the inclusion $\brun_4(D)\subset\brun_{1,3}(D)\subset \mathfrak{brun}_{2,2}(D)$.
\end{proof}

\section{$\mathfrak{dmr}_0$ is equivalent to $\pt^1_{dmr}(\varphi)=0$}

In this section, we present a characterization of the stuffle coproduct in terms of the pentagon equation map $\pt^1_{dmr}$ and prove that $\pt^1_{dmr}(\psi)=0$ is the same as original $\mathfrak{dmr}_0.$ The tools we use are mainly the polylogarithms calculation developed in \cite{Furusho2011,CS,howarth}.

\subsection{The Lie algebra $\mathfrak{dmr}_0$}
Let $\K\langle\langle X\rangle\rangle$, $\K\langle\langle Y\rangle\rangle$ denote the completed free associative algebra over the alphabet $X=\{x_0,x_1\}$ and $Y=\{y_n,n\in \Z_{\ge 1}\}$ with $\deg(x_0)=\deg(x_1)=1$, and $\deg(Y_n)=n$ and $y_0=1$. We recall the following notations.

\[
\begin{tikzcd}
\K\langle\langle X \rangle\rangle \arrow[r, shift left=0.5ex, "\pi_Y"] 
  \arrow[r, shift right=0.5ex, "\mathrm{corr}" below] 
& \K\langle\langle Y\rangle\rangle \arrow[r, "\Delta_{*}"] 
   & k\langle\langle Y\rangle\rangle^{\otimes 2}
\end{tikzcd}
\]

\begin{enumerate}
    \item $k\langle
    \langle Y\rangle\rangle$ injects into $\K\langle\langle X\rangle\rangle$ by the assignment $y_n\mapsto -x_0^{n-1}x_1$ where $n\ge 1$ and $y_0=1$, we identify $k\langle\langle Y\rangle\rangle$ with its image, the map $\pi_Y$ is defined as  the projection along $Ax_0$. For any $f\in \K\langle\langle X\rangle\rangle$,
    \begin{align*}
        \mathrm{corr}(f)=\sum_{n\ge 1}\frac{(-1)^n}{n}c_{x^{n-1}_0x_1}(f)y_1^n
    \end{align*}
    \item $\mathrm{corr}: \K\langle\langle X\rangle\rangle \to \K\langle\langle Y\rangle\rangle$ is defined to be the linear map taking the word $x^{n-1}_0x_1$ to $\frac{(-1)^{n}}{n}(y^n_1)$, in particular it takes 1 to 0.
    \item The map $\Delta_{*}:\K\langle\langle Y\rangle\rangle \to \K\langle\langle Y\rangle\rangle^{\otimes 2}$ is the unique algbera map determined by
    \begin{equation*}
    \Delta_{*}(y_n)=\sum^n_{i=0}y_i\otimes y_{n-i}
    \end{equation*}
\end{enumerate}

\begin{Def}
The \textit{double shuffle Lie algebra} $\mathfrak{dmr}_0$ introduced by Racinet in \cite{Racinet2002} is defined to be the set of formal Lie series $\psi\in \mathbb{L}^{\ge 3}(x_0,x_1)$ satisfying  
\begin{equation} \label{eq:def-dmr0}
	 \Delta_{*}(\psi_{*})=1\otimes \psi_{*}+\psi_{*}\otimes 1,
\end{equation}
where $\psi_{*}=(\psi)_{\text{corr}}+\pi_{Y}(\psi)$.    
\end{Def}

\subsection{Coproduct and counit sequence}
Let $\K\langle \langle Y\rangle\rangle\times_\K\K\langle\langle Y\rangle\rangle:=\{(b,c)\in \K\langle\langle Y\rangle\rangle\times \K\langle\langle Y\rangle\rangle\mid \varepsilon(b)=\varepsilon(c)\}$ denotes the fiber product, the map $\mathrm{diag}: \K\langle\langle Y\rangle \rangle\to \K\langle\langle Y\rangle\rangle \times_\K \K\langle\langle Y\rangle\rangle $ is the diagonal map
    \begin{equation*}
    \mathrm{diag}(a)=(a,a)
    \end{equation*}

We also denote by $\K\langle\langle Y\rangle\rangle_0:=\ker (\varepsilon|_{\K\langle\langle Y\rangle\rangle})$.

\begin{Lemma}\label{lem:dmr_short_exact_sequence}
    There is the split exact sequence of vector space
    \begin{align*}
        0\to \K\langle\langle Y\rangle\rangle_0\otimes \K\langle\langle Y\rangle\rangle_0\to \K\langle\langle Y\rangle\rangle\otimes \K\langle\langle Y\rangle\rangle\xrightarrow{q} \K\langle \langle Y\rangle\rangle\times_\K\K\langle\langle Y\rangle\rangle\to 0,
    \end{align*}
    where $q=(\id\otimes \varepsilon,\varepsilon\otimes \id).$ The sequence has the linear section $s(a,b):=a\otimes 1+1\otimes b-\varepsilon(a)1\otimes 1$.
\end{Lemma}

\begin{proof}
    By the direct sum $\K\langle\langle Y\rangle\rangle:=\K\oplus \K\langle\langle Y\rangle\rangle_0$, we have
    \begin{equation*}
    \K\langle\langle Y\rangle\rangle\otimes \K\langle\langle Y\rangle\rangle\simeq \K\oplus\K\langle\langle Y\rangle\rangle_0\oplus \K\langle\langle Y\rangle\rangle_0\oplus \K\langle\langle Y\rangle\rangle_0\otimes \K\langle\langle Y\rangle\rangle_0,
    \end{equation*}
    we also  the isomorphism
    \begin{equation*}
    \K\langle\langle Y\rangle\rangle\times_\K \K\langle\langle Y\rangle\rangle \simeq \K\oplus \K\langle\langle Y\rangle\rangle_0\oplus \K\langle\langle Y\rangle\rangle_0,
    \end{equation*}
    then the exactness follows. The map $s$ is a section by directly computation.
\end{proof}

\begin{Prop}\label{diagram_1}
We have the following commutative diagram in vector space
\begin{equation}\label{diagram_1}
\begin{tikzcd}
& & \K\langle\langle Y\rangle\rangle\arrow[d," \Delta_{*}" ]\arrow[rd,"\mathrm{diag}"]&&\\
0\arrow[r]&\K\langle\langle Y\rangle\rangle_0\otimes \K\langle\langle Y\rangle\rangle_0\arrow[r]& \K\langle\langle Y\rangle\rangle\otimes \K\langle\langle Y\rangle\rangle\arrow{r}{q}& \K\langle\langle Y\rangle\rangle\times_\K \K\langle\langle Y\rangle\rangle\arrow[r]&0\\
\end{tikzcd}
\end{equation}    
\end{Prop}

\begin{proof}
    By the counit identity, we have $q\circ\Delta^*=\mathrm{diag}$
\end{proof}

\subsection{Coproduct in $\mathcal{V}(\mathcal{M}_{0,4})$}
Let $\textbf{a}=(a_1,\dots,a_k)\in \Z^k_{>0}$, its \textit{weight} and its \textit{depth} are respectively defined to be $\text{wt}(\textbf{a})=a_1+\dots+a_k$ and $\text{dp}(\textbf{a})=k$, $\bf{a}$ is said to be \textit{admissible} if $a_k>1$. The multiple polylogarithms associated to $\bold a$ is 
\begin{equation}
\label{eq:polylogs-1-2}
	\Li_{(a_1, \ldots, a_k)}(z) :=\sum_{0<m_1<\dots<m_k}\frac{z^{m_k}}{m^{a_1}_1\dots m^{a_k}_k}
\end{equation}

By the result of \cite{Brown2009}, there is the embedding $\rho: V(\mathcal{M}_{0,4})  \hookrightarrow I_{o}(\widehat{\mathcal{M}_{0,4}})$ from the reduce bar construction $\mathcal{V}(\mathcal{M}_{0,4})$ to the iterated integrals. $\Li_{\bold a}(z)$ are in the image, we denote the pre-image of $\Li_{\bold a}(z)$ by $l_{\bold a}$, $\mathcal{V}(\mathcal{M}_{0,4})$ is the graded dual of $\K\langle\langle X\rangle\rangle$, it is cogenerated by $H^1(\mathcal{M}_{0,4})$ spanned by the forms  $\omega_0=\frac{dz}{z}$,  $\omega_1=\frac{dz}{z-1}$, and
\begin{equation*}
	l_{\bf{a}}=(-1)^k[\underbrace{\omega_0|\dots|\omega_0|}_{a_k-1}\omega_1|\underbrace{\omega_0|\dots |\omega_0}_{a_{k-1}-1}|\omega_1|\dots |\underbrace{\omega_0|\dots |\omega_0|}_{a_1-1}\omega_1],
\end{equation*}

For $\dep(\bold a)=k, \dep(\bold b)=l$, the stuffle product of $l_{\bold a}$'s on $\mathcal{V}(\mathcal{M}_{0,4})$ is defined to be 
\begin{align*}
    l_{\bold a}\shuffle_s l_{\bold b}:=\sum_{\substack{\sigma\in \text{Sh}^{\le(k,l)}, \\
        \sigma^{-1}(N)=k,k+l}} l_{\sigma(\bf{a},\bf{b})} 
        +\sum_{\substack{\sigma\in \text{Sh}^{\le(k,l)}, \\ \sigma^{-1}(N)=k+l}} l_{\sigma(\bf{a},\bf{b})}
        +\sum_{\substack{\sigma\in \text{Sh}^{\le(k,l)}, \\ \sigma^{-1}(N)=k}} l_{\sigma(\bf{a},\bf{b})}.
\end{align*}
\begin{align*}
    l_{\emptyset}\shuffle_s l_{\bold b}:=l_{\bold b},\quad l_{\bold a}\shuffle_s l_{\emptyset}:= l_{\bold a}.
\end{align*}

where 
\begin{equation*}
\Sh^{\leq(k,l)} = \bigcup_{N=1}^{\infty} \Biggl\{ 
       \begin{array}{l|cl}
                      & \sigma \text{ is onto;}\\
             \sigma : \{1, \ldots, k+l\} \to \{1, \ldots, N\} & \sigma(1) < \ldots < \sigma(k);\\
             & \sigma(k+1) < \ldots < \sigma(k+l)
        \end{array}
     \Biggr\},
\end{equation*}
and
\begin{equation*}
\begin{split}
    \sigma(x,y) & := \begin{cases}
        xy, \hfill & \text{ if } \sigma^{-1}(N)=k, k+l; \\
        (x,y), \hfill & \text{ if } \sigma^{-1}(N)=k+l; \\
        (y,x), \hfill & \text{ if } \sigma^{-1}(N)=k,
    \end{cases}\\
    \sigma(\textbf{a}, \textbf{b}) & :=((c_1,\ldots, c_j),(c_{j+1}, \ldots, c_N)), 
\end{split}
\end{equation*}
where
    \begin{align*}
        \{j,N\} & =\{\sigma(k), \sigma(k+l)\},\\
        c_i & = \begin{cases}
            a_s + b_{t-k}, \hfill & \text{ if } \sigma^{-1} = \{s,t\} \text{ with } s < t;\\
            a_s, \hfill & \text{ if } \sigma^{-1} = \{s\} \text{ with } s \leq k;\\
            b_{s-k}, \hfill & \text{ if } \sigma^{-1} = \{s\} \text{ with } s > k.\\
        \end{cases}
    \end{align*}

The vector subspace of $\mathcal{V}(\mathcal{M}_{0,4})$ spanned by $l_{\bold a}$ is graded dual to $\K\langle\langle Y\rangle\rangle$ and the coproduct $\Delta_{*}$ is dual to the stuffle product $\shuffle_s$,
\begin{equation}
    \langle \Delta_{*}(f),l_{\bold a}\otimes l_{\bold b}\rangle=\langle f, l_{\bold a}\shuffle_s l_{\bold b}\rangle, \quad f\in \K\langle\langle Y\rangle\rangle, \quad \forall\bold a,\bold b
\end{equation}

We now introduce the vector subspace $\mathcal{P}$ of $\K\langle\langle x,y,z\rangle\rangle
$ which is defined as $\K$ linear span of the following basis for any $\bold a=(a_1,\ldots,a_k)$ and $\bold b=(b_1,\ldots,b_n)$ 
 \begin{align*}
   &p_{\bold a, \bold b}=(-1)^{k+n}z^{b_n-1}x\ldots z^{b_1-1}xz^{a_k-1}yz^{a_{k-1}-1}y\ldots z^{a_1-1}y,\\
   &p_{\emptyset,\bold b}=(-1)^nz^{b_n-1}x\ldots z^{b_1-1}x,\quad p_{\bold a, \emptyset}=(-1)^kz^{a_k-1}yz^{a_{k-1}-1}y\ldots z^{a_1-1}y,\\
   &p_{\emptyset,\emptyset}=1.
 \end{align*}

 $\mathcal{P}$ is a graded vector space with respect to the degree of $x$ and $y$,
 \begin{align*}
     \mathcal{P}&=\mathcal{P}_{0,0}\oplus \mathcal{P}_{\ge 1,0}\oplus \mathcal{P}_{0,\ge 1}\oplus \mathcal{P}_{\ge 1,\ge 1}\\
     &=\K\oplus \mathrm{span}(p_{\emptyset,\bold b})\oplus \mathrm{span}(p_{\bold a,\emptyset})\oplus \mathrm{span}(p_{\bold a,\bold b})\\
     &=\K\oplus \K\langle x,z\rangle x\oplus \K\langle y,z\rangle y\oplus \mathcal{P}_{\ge 1,\ge 1}.
 \end{align*}

\begin{Lemma}
The vector space $\mathcal{P}$ is isomorphic to $(\K\cdot 1+\K\langle X\rangle x_1)\otimes (\K\cdot 1+\K\langle X\rangle x_1)$, for any $\bold a.\bold b$,
\begin{align*}
    p:(\K\cdot 1+ \K\langle
    \langle X\rangle\rangle x_1)\otimes (\K\cdot 1+\K\langle\langle X\rangle\rangle x_1)&\to \mathcal{P},\\
    (-1)^{k+n}x^{a_k-1}_0x_1\ldots x_0^{a_1-1}x_1\otimes x^{b_n-1}_0x_1\ldots x_0^{b_1-1}x_1&\mapsto p_{\bold a,\bold b},\\
    (-1)^k x^{a_k-1}_0x_1\ldots x_0^{a_1-1}x_1\otimes 1&\mapsto p_{\bold a,\emptyset},\\
    (-1)^n 1\otimes x^{b_n-1}_0x_1\ldots x_0^{b_1-1}x_1&\mapsto p_{\emptyset,\bold b}.
\end{align*}
\end{Lemma}

Through the identification $p$, the vector space $1\otimes \K\langle\langle X\rangle\rangle x_1$ is isomorphic to $\mathcal{P}_{\ge 1,0}$, the vector space $ \K\langle\langle X\rangle\rangle x_1\otimes 1$ is isomorphic to $\mathcal{P}_{0,\ge 1},$ the vector space $\K\cdot 1$ is identified with $\mathcal{P}_{0,0}$. We introduce the projection map
\begin{equation*}
p_y:\mathcal{P}\to \mathcal{P}_{0,\ge 0},\quad p_x:\mathcal{P}\to \mathcal{P}_{\ge 0,0}.    
\end{equation*}

Using the isomorphism $p$, we have the identification $\mathcal{P}_{\ge 1,\ge 1}\simeq \K\langle\langle Y\rangle\rangle_0\otimes \K\langle\langle Y\rangle\rangle_0$, we could transport the diagram in the Proposition \ref{diagram_1} into the following. 

\begin{Prop}\label{diagram_2}
The following diagram commutes as vector spaces and $\delta_{\mathcal{P}}(h):=(p(h\otimes 1),p(1\otimes h))$ for $h$ in $\K\langle\langle Y\rangle\rangle$.

\begin{equation}\label{diagram_1}
\begin{tikzcd}
& & \K\langle\langle Y\rangle\rangle\arrow[d,"p\circ \Delta_{*}" ]\arrow[rd,"\delta_{\mathcal{P}}"]&&\\
0\arrow[r]&\mathcal{P}_{\ge 1,\ge 1}\arrow[r]& \mathcal{P}\arrow{r}{(p_y,p_x)}& \mathcal{P}_{0,\ge 0}\times_{\K} \mathcal{P}_{\ge 0,0}\arrow[r]&0\\
\end{tikzcd}
\end{equation}    
\end{Prop}

\subsection{Coproduct in $\mathcal{V}(\mathcal{M}_{0,5})$}
For $\bold a,\bold b$ with $\dep(\bold a)=k$ and $\dep(\bold b)=l$, we could associated to the pair $(\bold a,\bold b)$ a two variable polylogarithms function,
\begin{equation}
\label{eq:polylogs-2}
    \Li_{(a_1, \ldots, a_k),(b_1, \ldots, b_l)}(x,y) :=\sum_{\substack{0<m_1<\dots< m_k\\ <n_1<\dots <n_l}}\frac{x^{m_k}y^{n_l}}{m^{a_1}_1\dots m^{a_k}_kn^{b_1}_1\dots n^{b_l}_l}.
\end{equation}
through the injective map $\rho: V(\mathcal{M}_{0,5})  \hookrightarrow I_{o}(\widehat{\mathcal{M}_{0,5}})$, the following bar words are introduced in \cite{Furusho2011},
\begin{align*}
    l^{x,y}_{\bf{a},\bf{b}} & :=\rho^{-1}(\Li_{\bold a,\bold b}(x,y)), \quad l^{y,x}_{\bf{a},\bf{b}} :=\rho^{-1}(\Li_{\bold a,\bold b}(y,x)), \\
    l^{x}_{\bf{a}} & :=\rho^{-1}(\Li_{\bold a}(x)), \quad l^{y}_{\bf{a}} :=\rho^{-1}(\Li_{\bold a}(y)), \\
    l^{xy}_{\bf{a}} & :=\rho^{-1}(\Li_{\bold a}(xy)).
\end{align*}

The reduced bar construction $\mathcal{V}(\mathcal{M}_{0,5})$ is the subspace of the tensor coalgebra $H^1({\mathcal{M}_{0,5}})$ satisfying the Chen's integrability relations, $H^1(\mathcal{M}_{0,5})$ is spanned by the forms $\omega_{45}$,$\omega_{12}$,$\omega_{23}$,$\omega_{24}$,$\omega_{34}$. $\mathcal{V}(\mathcal{M}_{0,5})$ is graded dual to $U\mathfrak{p}_5$ through the identification of the degree 1 part through the identification of the degree $1$ part $\omega_{45},\omega_{34},\omega_{24},\omega_{12},\omega_{23}$ with $X_{45},X_{34}, X_{24},X_{12}, X_{23}$, for more details see \cite{Furusho2011,howarth}. The two variable formal polylogarithms has the following relations.

\begin{Lemma}[\cite{Furusho2011}, equation (3.2)] Let $\varphi\in \mathfrak{p}_5$ be a Lie series, then the following \emph{series shuffle relations (stuffle)} modulo product holds for any index $\bf{a},\bf{b}$ with $\dep(\bold a)=k, \dep(\bold b)=l,$
	\begin{equation}\label{eq:stuffle 2 variable}
		\sum_{\substack{\sigma\in \text{Sh}^{\le(k,l)}, \\
        \sigma^{-1}(N)=\{k,k+l\}}} l^{xy}_{\sigma(\bf{a},\bf{b})}(\varphi) 
        +\sum_{\substack{\sigma\in \text{Sh}^{\le(k,l)}, \\ \sigma^{-1}(N)=\{k+l\}}} l^{x,y}_{\sigma(\bf{a},\bf{b})}(\varphi) 
        +\sum_{\substack{\sigma\in \text{Sh}^{\le(k,l)}, \\ \sigma^{-1}(N)=\{k\}}} l^{y,x}_{\sigma(\bf{a},\bf{b})}(\varphi)=0.
	\end{equation}
\end{Lemma}

For $\psi\in \mathbb{L}(x_0,x_1)$, denote $\psi(X_{ij},X_{jk})\in \mathfrak{p}_5$ by $\psi_{ijk}$ for $1\le i,j,k\le 5$ and
$c_{w}(\psi)$ the coefficient of the word $w$ in $\varphi$. Recall the basic relations between two and one variable polylogarithms of Lemmas $4.1$ and $4.2$ from \cite{Furusho2011} and Lemma $3$ from \cite{CS}.

\begin{Lemma}[{\cite[Lemma 4.1,4.2]{Furusho2011},\cite[Lemma 3]{CS}}]
\label{lemma:polylogs-compilation}
Let $\psi \in \mathbb{L}(x_0,x_1)$ be a Lie series. Then,
\begin{enumerate}
    \item \label{lemma:543}$l^{y,x}_{\bf{a},\bf{b}}(\psi_{543}) =0,\quad \text{for any}~ {\bf{a},\bf{b}};$
    \item \label{lemma:215} $l^{y,x}_{\bf{a},\bf{b}}(\psi_{215}) =l_{\bf{a}\bf{b}}(\psi),\quad \text{for any}~ {\bf{a},\bf{b}}; $
    \item \label{lemma:432} $l^{y,x}_{\bf{a},\bf{b}}(\psi_{432}) =0,\quad \text{for}~ {\bf{a},\bf{b}} \ne (1,\dots,1), (1,\dots,1);$
    \end{enumerate}
\begin{enumerate}\setcounter{enumi}{3}
    \item \label{eq: 5 term_1} $l^{xy}_{\bf{a}}(\psi_{451}+\psi_{123})=l_{\bf{a}}(\psi),\quad \text{for any}~ {\bf{a},\bf{b}};$
    \item \label{eq: 5 term_2} $l^{x,y}_{\bf{a},\bf{b}}(\psi_{451}+\psi_{123})=l_{\bf{a}\bf{b}}(\psi),\quad \text{for any}~ {\bf{a},\bf{b}}.$
\end{enumerate}
\end{Lemma}

\begin{Prop}\label{prop:coproduct}
For any $\psi\in \mathbb{L}(x_0,x_1)$ and $\bold a,\bold b$ with $\dep(a)=k,\dep(b)=l$, we have the relation
\begin{equation*}
    l_{\bold a}\shuffle_s l_{\bold b}(\psi)=-\sum_{\substack{\sigma\in \text{Sh}^{\le(k,l)}, \\ \sigma^{-1}(N)=\{k\}}} l^{y,x}_{\sigma(\bf{a},\bf{b})}(\psi_{451}+\psi_{123}-\psi_{215}-\psi_{543}).
\end{equation*}
\end{Prop}

\begin{proof}
We evaluate the equation \eqref{eq:stuffle 2 variable} at the element $\varphi=\psi_{451}+\psi_{123}$, then we have the relation
\begin{align*}
&\sum_{\substack{\sigma\in \text{Sh}^{\le(k,l)}, \\
        \sigma^{-1}(N)=k,k+l}} l^{xy}_{\sigma(\bf{a},\bf{b})}(\varphi) 
        +\sum_{\substack{\sigma\in \text{Sh}^{\le(k,l)}, \\ \sigma^{-1}(N)=k+l}} l^{x,y}_{\sigma(\bf{a},\bf{b})}(\varphi) 
        +\sum_{\substack{\sigma\in \text{Sh}^{\le(k,l)}, \\ \sigma^{-1}(N)=k}} l^{y,x}_{\sigma(\bf{a},\bf{b})}(\varphi)\\
&=\sum_{\substack{\sigma\in \text{Sh}^{\le(k,l)}, \\
        \sigma^{-1}(N)=\{k,k+l\}}} l_{\sigma(\bf{a},\bf{b})}(\psi) 
        +\sum_{\substack{\sigma\in \text{Sh}^{\le(k,l)}, \\ \sigma^{-1}(N)=\{k+l\}}} l_{\sigma(\bf{a},\bf{b})}(\psi) 
        +\sum_{\substack{\sigma\in \text{Sh}^{\le(k,l)}, \\ \sigma^{-1}(N)=\{k\}}} l^{y,x}_{\sigma(\bf{a},\bf{b})}(\psi_{451}+\psi_{123})=0        
\end{align*}

The results then follows from the equality (1) and (2) in Lemma \ref{lemma:polylogs-compilation}.
\end{proof}

\begin{Lemma}[{\cite[Lemma 15]{howarth}}]
\label{lemma:l-kerpr2}
For all indices $\textbf{a,b}=(a_1, \ldots, a_k),(b_1, \ldots, b_l)$,
\begin{equation}
    \label{eq:l-kerpr2}
	\eval{l^{y,x}_{{\bf{a}},{\bf{b}}}}_{\ker \spr^2_5}=(-1)^{k+l}w^{b_l-1}_{12}w_{23}\dots \omega_{12}^{b_1-1}\omega_{23}\omega^{a_k-1}_{12}\omega_{24}\dots \omega^{a_1-1}_{12}\omega_{24},
\end{equation}
which means that for any $\alpha\in \ker \rm{pr}_2$, \[l^{y,x}_{\bold a,\bold b}(\alpha)=(-1)^{k+l}w^{b_l-1}_{12}w_{23}\dots \omega_{12}^{b_1-1}\omega_{23}\omega^{a_k-1}_{12}\omega_{24}\dots \omega^{a_1-1}_{12}\omega_{24}(\alpha).\]
\end{Lemma}

\begin{Lemma}\label{lemma: 1-1}
For any $\psi\in \mathbb{L}^{\ge 2}(x_0,x_1)$,
\begin{align*}
    l^{y,x}_{\bold a,\bold b}(\mathrm{pent}(\psi(X_{12},X_{23}))=0,\quad \bold a=(1,\ldots,1),\bold b=(1,\ldots,1)
\end{align*}  
\end{Lemma}

\begin{proof}
As $\mathrm{pent}(\psi(X_{12},X_{23}))\in \ker \spr^1_5\cap \ker \spr^2_5=(X_{12}) $, when restricting to the space $\bold a=(1,\ldots,1),\bold b=(1,\ldots,1)$, the words $l^{y,x}_{\bold a,\bold b}$ has no words in $\omega_{12}$ by the Lemma \ref{lemma:l-kerpr2}, the results then follows.
\end{proof}

Let $\Lambda_{\bold a,\bold b}:=\sum_{\substack{\sigma\in \text{Sh}^{\le(k,l)}, \\ \sigma^{-1}(N)=k}} l^{y,x}_{\sigma(\bf{a},\bf{b})}$ for any $\bold a,\bold b$ where $k=\dep(\bold a),l=\dep(\bold b).$

\begin{Prop}\label{prop:coprod_pent}
    For any Lie series $\eta\in \mathbb{L}^{\ge 2}(x_0,x_1)$, the following relation hold
    \begin{align}
    \langle\Delta_{*}(\pi_Y(\eta)), l_{\bold a}\otimes l_{\bold b}\rangle=
    \begin{cases}
    \langle \eta, l_{\bold b}\rangle,\quad \bold a=\emptyset;\\
    \langle \eta, l_{\bold a}\rangle, \quad \bold b=\emptyset;\\
    \langle -\mathrm{pent}(\eta)+\eta_{234}, \Lambda_{\bold a,\bold b}\rangle,
    \qquad \bold a\ne \emptyset,\bold b \ne \emptyset;
    \end{cases}   
    \end{align}.
\end{Prop}

\begin{proof}
By the relations in $\mathfrak{p}_n$, We have the relations that
\begin{equation*}
\pt(\eta)=\eta_{451}+\eta_{123}+\eta_{234}-\eta_{215}-\eta_{543}.
\end{equation*}
We distinguish two cases for $\bold a,\bold b\ne \emptyset$. For $\bold a,\bold b \ne (1,\ldots,1),(1,\ldots,1)$, the results follows from the Lemma \ref{lemma:polylogs-compilation} and the Proposition \ref{prop:coproduct}. For $\bold a,\bold b= (1,\ldots,1),(1,\ldots,1)$, the results follows from Lemma \ref{lemma: 1-1}.
\end{proof}

Because $\mathrm{pent}(\eta)\in \ker \mathrm{spr}_2$ and $\eta_{234}\in \ker \mathrm{spr}_2$, we have
\begin{align*}
 \langle \mathrm{pent}(\eta), \Lambda_{\bold a,\bold b}\rangle=\langle \mathrm{pent}(\eta), \Lambda_{\bold a,\bold b}|_{\ker \spr^2_5}\rangle, \quad   \langle \eta_{234}, \Lambda_{\bold a,\bold b}\rangle=\langle \eta_{234}, \Lambda_{\bold a,\bold b}|_{\ker \spr^2_5}\rangle,
\end{align*}
therefore it is suffice to express the coproduct in $\ker \mathrm{spr}^2_5$ by Proposition \ref{prop:coproduct}.  By the same arguments as  Proposition \ref{prop:lambda_13} after permuting the strands, 
\[K_4: \ker \pr^2_4\simeq \mathbb{L}(t_{12},t_{23},t_{24})\to \ker \mathrm{spr}^2_5\simeq \mathbb{L}(X_{12},X_{23},X_{24}),\]
$K_4$ is an isomorphism, so it suffice to consider $\ker \pr^2_4$ and $\mathfrak{t}_4$ in the following. For simplification of the notation, we denote 
\[I:=\brun_{1,3}=\ker \mathrm{pr}_2\cap \ker \mathrm{pr}_3\cap \ker \mathrm{pr}_4=[(t_{23}),t_{24})]\subset \ker \mathrm{pr}_2\] in the following. 

Let $\rho:\K\langle\langle t_{12},t_{23},t_{24}\rangle\rangle\simeq \K\langle\langle z,x,y\rangle\rangle$ be the algebra isomorphism determined by $t_{12}\mapsto z,t_{23}\mapsto x,t_{24}\mapsto y$, let $\mathrm{pr}_{\mathcal{P}}$ be the linear projection on $\mathcal{P}$, we could define
\begin{equation*}
    \pi:\mathrm{pr}_{\mathcal{P}}\circ \rho,
\end{equation*}
it restricts to a map $\mathbb{L}(t_{12}, t_{23},t_{24})\to \mathcal{P}$. By the specific order of $x,y$ in the definition of $\mathcal{P}$, we have that $\pi([I,I])=0$, it induces the maps
\begin{align*}
\bar{\pi}:\mathbb{L}(t_{12},t_{23},t_{24})/[I,I]\to \mathcal{P},\quad \bar{\pi}_I:=\bar{\pi}\mid_I:I/[I,I]\to \mathcal{P}_{\ge 1,\ge 1}.    
\end{align*}

\begin{Lemma}[{\cite[Theorem 6.4]{kernel}}]\label{lem:basis_kernel}
     $\bar{\pi}_I: I/[I,I] \simeq \mathcal{P}_{\ge 1, \ge 1}\simeq \K\langle \langle Y\rangle\rangle_0\otimes \K\langle\langle Y\rangle\rangle_0$.
\end{Lemma}

For $v\in I$, denote $\bar{v}\in I^{\mathrm{ab}}=I/[I,I]$, the adjoint actions $\ad_{t_{23}}$ and $\ad_{t_{24}}$ passing to the quotient, and the actions are commute in $I^{\mathrm{ab}}$. Therefore the vector space $I/[I,I]$ contains the vector subspace 
\begin{align}\label{eq:W_I}
W_I:=\{f(\ad_{t_{23}},\ad_{t_{24}})\overline{[t_{23},t_{24}]}\mid f\in \K[[x,y]] \}.
\end{align}

The vector space $W_I$ is spanned by the basis $\tilde{p}_{k,l}:=\ad(t_{23})^{k-1}\ad(t_{24})^{l-1}(\overline{[t_{23},t_{24}]})$ and its image under $\bar{\pi}_I$ is
\begin{align*}
    &\bar{\pi}_I: \tilde{p}_{k,l}\mapsto (-1)^{k-1} p_{\underbrace{(1,\ldots,1)}_{\dep=l},\underbrace{(1,\ldots,1)}_{\dep=k}},
\end{align*}
the modified abelianization map is defined as
\begin{align*}
&(-)_I^{\mathrm{ab}}:\mathbb{L}^{\ge 2}(x_0,x_1)\to W_I\subset I/[I,I],\qquad \psi\mapsto \overline{\psi(t_{23},-t_{23}-t_{24})}.
\end{align*}

For non empty $\bold a,\bold b,\bold c,\bold d$, recall that $\mathcal{P}_{\ge 1,\ge 1}=\mathrm{span}\{p_{\bold a,\bold b}\}$, we also denote $l_{\bold c,\bold d}:=l^{y,x}_{\bold a,\bold b}|_{\K\langle x_{12},x_{23},x_{24}\rangle}$, by the duality Lemma \ref{lemma:l-kerpr2}, we have
\begin{equation}
    \langle p_{\bold a,\bold b},l_{\bold c,\bold d}\rangle=\delta_{\bold a,\bold c}\delta_{\bold b,\bold d}.
\end{equation}

We now define the linear map $\lambda:\mathcal{P}\to \mathcal{P}$. $\lambda$ is identity on the part $\mathcal{P}_{0,0}$, $\mathcal{P}_{\ge 1,0}$ and $\mathcal{P}_{0,\ge 1}$, on the part $ \mathcal{P}_{\ge 1,\ge 1}$,  it is defined to be
\begin{equation}
    \langle \lambda(p_{\bold a,\bold b}),l_{\bold c,\bold d}\rangle:=\langle p_{\bold a,\bold b},\sum_{\substack{\sigma\in \text{Sh}^{\le(r,s)}, \\ \sigma^{-1}(N)=k}} l_{\sigma(\bf{c},\bf{d})}\rangle=\langle p_{\bold a,\bold b},\Lambda_{\bold c,\bold d}\rangle;\qquad r=\dep(\bold c),s=\dep(\bold d).
\end{equation}
which is simply the graded transpose of the map $\Lambda_{\bold c,\bold d}$.

For nonempty $\bold a,\bold b,\bold c,\bold d$, define
\begin{equation*}
    N^{\bold c,\bold d}_{\bold a,\bold b}:=\{\sigma\in \text{Sh}^{\le (r,s)}\mid \sigma^{-1}(N)=\{r\},\sigma(\bold c,\bold d)=(\bold a,\bold b)\},
\end{equation*}
where $r=\dep(\bold c)$ and $s=\dep(\bold d)$, then we have the explicit formula by the definition of the graded transpose
\begin{equation}
\lambda(p_{\bold a,\bold b)})=\sum_{\bold c,\bold d\ne \emptyset}|N^{\bold c,\bold d}_{\bold a,\bold b}|p_{\bold c,\bold d }, \quad p_{\bold a,\bold b} \in \mathcal{P}_{\ge 1,\ge 1}.
\end{equation}

\begin{Prop}\label{prop: stuff_diagram}
    The following diagram commutes
\begin{equation}\label{prop:dmr_diagram}
\begin{tikzcd}
\mathbb{L}^{\ge 2}(x_0,x_1) \arrow{r}{\pi_Y}\arrow{d}{-(\mathrm{pent}^1_{dmr}-(-)_I^{\mathrm{ab}})} &[6em] k\langle\langle Y\rangle\rangle\arrow{dd}{\Delta_{*}-1\otimes (-)-(-)\otimes 1}\\
I/[I,I]\arrow{d}{\lambda\circ \bar{\pi}_I} &\\
\mathcal{P}_{\ge 1,\ge 1}\arrow[r,"p^{-1}"]& k\langle\langle Y\rangle\rangle_0\otimes k\langle\langle Y\rangle\rangle_0.
\end{tikzcd}    
\end{equation}
\end{Prop}

\begin{proof}
By the Proposition \ref{prop:coprod_pent}, we have the relation $\langle{\Delta_*^\mathrm{red}}(\pi_Y(\eta)),l_{\bold a}\otimes l_{\bold b}\rangle=\langle -\pt(\eta)+\eta_{234},\Lambda_{\bold a,\bold b}\rangle$. By the definitions of $(-)_I^{\mathrm{ab}}$ and $\lambda$, they replace the $\eta_{234}$ and $\Lambda_{\bold a,\bold b}$ part.
\end{proof}

Let us introduce the map
\begin{align*}
\mathrm{diag}^{123,124}: \mathbb{L}(x_0,x_1)\to \ker \mathrm{pr}_2,\quad \psi(x_0,x_1)\mapsto \psi(X_{12},X_{23})+\psi(X_{12},X_{24}).    
\end{align*}

\begin{Th}\label{th:shuffle coproduct}
The following diagrams commutes

\begin{equation}\label{dmr_diagram_2}
\begin{tikzcd}
\mathbb{L}^{\ge 2}(x_0,x_1) \arrow{r}{\pi_Y}\arrow{d}{(\mathrm{diag}^{123,124}-(\mathrm{pent}^1_{dmr}-(-)_I^{\mathrm{ab}})} &[10em] \K\langle\langle Y\rangle\rangle\arrow{dd}{\Delta_{*}}\\
\mathbb{L}(t_{12},t_{23},t_{24})/[I,I]\arrow{d}{\lambda\circ \bar{\pi}} &\\
\mathcal{P}\arrow[r,"p^{-1}"]& \K\langle\langle Y\rangle\rangle\otimes \K\langle\langle Y\rangle\rangle.
\end{tikzcd}
\end{equation}
\end{Th}

\begin{proof}
The theorem follows from the Proposition \ref{prop: stuff_diagram}.

\end{proof}

We have the following short exact sequence similar to that in Lemma \ref{lem:dmr_short_exact_sequence},
\begin{align*}
    0\to I/[I,I]\to \mathbb{L}(t_{12},t_{23},t_{24})/[I,I]\xrightarrow{(\mathrm{pr}_3,\mathrm{pr}_4)}\mathbb{L}(t_{12},t_{24})\times_{\K
    t_{12}} \mathbb{L}(t_{12},t_{23})\to 0,
\end{align*}  
it split as a vector space, the proof is similar to Lemma \ref{lem:dmr_short_exact_sequence}. Based on the  sequence, let us introduce the map,
\begin{align*}
   &\tc:\mathbb{L}(t_{12},t_{23},t_{24})/[I,I]\to\mathcal{P},\quad v\mapsto p\circ \Delta_*((\pr_3(v))_{\mathrm{corr}}+(\pr_4(v))_{\mathrm{corr}})/2
\end{align*} 

\begin{Cor}
The following diagram commutes
\begin{equation}\label{dmr_diagram_3}
\begin{tikzcd}
\mathbb{L}^{\ge 2}(x_0,x_1) \arrow{r}{\pi_Y+(-)_{\mathrm{corr}}}\arrow{d}{(\mathrm{diag}^{123,124}-(\mathrm{pent}^1_{dmr}-(-)_I^{\mathrm{ab}})} &[10em] \K\langle\langle Y\rangle\rangle\arrow{dd}{\Delta_{*}}\\
\mathbb{L}(t_{12},t_{23},t_{24})/[I,I]\arrow{d}{\lambda\circ \bar{\pi}+\tc} &\\
\mathcal{P}\arrow[r,"p^{-1}"]& \K\langle\langle Y\rangle\rangle\otimes \K\langle\langle Y\rangle\rangle.
\end{tikzcd}
\end{equation}
\end{Cor}

\subsection{Proof of $\mathfrak{dmr}_0$ is equivalent to $\pt^1_{dmr}$}

We first recall the theorem.

\begin{Th}[{\cite[Theorem A]{howarth}}]\label{th:HR}
	Let $\psi\in \mathbb{L}(x_0,x_1)$ be such that $c_{x_0}(\psi)=c_{x_1}(\psi)=c_{x_0x_1}(\psi)=0$, then the following two conditions are equivalent:
	\begin{enumerate}[label=(\roman*)] 
		\item $\psi\in \mathfrak{dmr}_0$;
		\item $l^{y,x}_{\bf{a},\bf{b}}(\psi_{451}+\psi_{123}-\psi_{432}-\psi_{215}-\psi_{543})=0,$ for ${\bf{a},\bf{b}}\ne (1,\dots,1),(1,\dots,1)$.
	\end{enumerate}
\end{Th}

\begin{Lemma}\label{lem:proof_dmr}
    For $\psi\in \mathbb{L}^{\ge 3}(x_0,x_1)$,$\psi\in \mathfrak{dmr}_0$ if and only if $l^{y,x}_{\bold a,\bold b}(\pt(\psi(X_{12},X_{23})))=0 $ for all $\bold a,\bold b$.
\end{Lemma}

\begin{proof}
    We have that 
    \begin{align*}
    &l^{y,x}_{\bold a,\bold b}(\pt(\psi(X_{12},X_{23}))-l^{y,x}_{\bf{a},\bf{b}}(\psi_{451}+\psi_{123}-\psi_{432}-\psi_{215}-\psi_{543})\\
    &=l^{y,x}_{\bold a,\bold b}(\psi_{234}+\psi_{432})
    \end{align*}
By the Lemma \ref{lemma:polylogs-compilation},  for ${\bf{a},\bf{b}}\ne (1,\dots,1),(1,\dots,1)$, $l^{y,x}_{\bold a,\bold b}(\psi_{432}+\psi_{234})=0$. Therefore we have the equivalence of $\mathfrak{dmr}_0$ with $^{y,x}_{\bold a,\bold b}(\pt(\psi(X_{12},X_{23}))$ for $\bold a,\bold b\ne (1,\ldots,1),(1,\ldots,1)$ by the Theorem \ref{th:HR}.

By Lemma \ref{lemma: 1-1}, when $\bold a,\bold b=(1,\ldots,1),(1,\ldots,1)$, $l^{y,x}_{\bold a,\bold b}(\pt(\psi))=0$, therefore the result in this Lemma follows.
\end{proof}

\begin{Th}\label{th:dmr}
   Let $\psi\in \mathbb{L}^{\ge 3}(x_0,x_1)$ $\psi\in \mathfrak{dmr}_0$ if and only if $\pt^1_{dmr}(\psi)=0$
\end{Th}

\begin{proof}
    By the Theorem 6.4 in \cite{kernel}, the kernel of the polylogarithms $l^{y,x}_{\bold a,\bold b}\mid_I$ is $[I,I]$, which means that $l^{y,x}_{\bold a,\bold b}(\eta)=0$, for any $\bold a,\bold b$,  if and only if $\eta\in [I,I]$. Therefore by the Lemma \ref{lem:proof_dmr}, we have $\psi\in \mathfrak{dmr}_0$ if and only if $\pt(\psi(X_{12},X_{23}))=0$ in $I/[I,I]$.  
    
    By the Proposition \ref{prop:lambda_13}, $\pt(\psi(X_{12},X_{23}))=0$ in $I/[I,I]$ is equivalent to $\pt(\psi(t_{12},t_{23}))=0$ in $\brun_{1,3}(D)/\Gamma_2(\brun_{1,3}(D))$, therefore the result follows.
\end{proof}

\section{$\mathfrak{krv}_2$ and pentagon defect}
\label{sec:krv}

We relate the Brunnian pentagon defect to the reduced coaction,
divergence, and the necklace cobracket.

Let $x_{\infty}:=-x_0-x_1$,
$A=\K\langle\langle x_0,x_1\rangle\rangle$ and $L:=\mathbb{L}(x_0,x_1)$. 

\subsection{Relating divergence and the cobracket}
In this subsection, we recall the relation between divergence and necklace cobracket discoved in \cite{GTgenus0}.

We recall that the space of tangential derivations are
$\operatorname{tDer}(A)=A\oplus A$. For $u=(a,b)$, set
$\rho(u)(x_0)=[x_0,a]$ and $\rho(u)(x_1)=[x_1,b]$.
The space of tangential derivations is a Lie algebra with the Lie bracket $\{u,v\}$
\[
 \{u,v\}_i=\rho(u)(v_i)-\rho(v)(u_i)+[u_i,v_i],
\]
$\rho$ is a Lie algebra morphism from $\mathrm{tDer}(A)$ to the Lie algebra of algebra derivations $\mathrm{Der}(A)$, and the special derivations are
$\operatorname{sDer}(A)=\{(a,b):[x_0,a]+[x_1,b]=0\}$ is a Lie subalgebra.
We write $\mathfrak{tder}_2=L\oplus L$ and
$\mathfrak{sder}_2=\operatorname{sDer}(A)\cap (L\oplus L)$.

Let $|A|=A/[A,A]$ denote the cyclic words. For a word of length $r>0$, define the symmetrization map
\[
 N(|s_1\cdots s_r|)
   =\sum_{j=1}^r s_j\cdots s_rs_1\cdots s_{j-1},
 \qquad N(|1|)=0.
\]
Then $|N(c)|=rc$ in degree $r$. The action of $\mathrm{Der}(A)$ on $A$ descends to an action on $|A|$.

We use the following decompositions of $f\in A$, 
\[
 f=\varepsilon(f)+x_0 d_{x_0}^Rf+x_1 d_{x_1}^Rf
  =\varepsilon(f)+(d_{x_0}^Lf)x_0+(d_{x_1}^Lf)x_1.
\]

\begin{Lemma}[{\cite[Lemma~8.3]{GTgenus0}}]
\label{lem:s5-H}
The map
\[
 H:|A|/\K\cdot1 \longrightarrow\operatorname{sDer}(A), \quad c\mapsto
 H(c)=(-d_{x_0}^RN(c),-d_{x_1}^RN(c)),
\]
is an isomorphism of two vector space. For homogeneous $u=(a,b)$ of component degree
$m\geq0$, its inverse is
\[
 h(u):=H^{-1}(u)=-\frac1{m+1}|x_0a+x_1b|.
\]
\end{Lemma}

Using the isomorphism $H$ and the Lie algebra structure of $\mathrm{sDer}(A)$, the vector space $|A|/\K \cdot 1$ is endowed with a Lie bracket, we call it the necklace Lie bracket $[-,-]_{\mathrm{necklace}}$ and necklace Lie algebra.

The divergence map $\mathrm{div}:\mathrm{tDer}(A)\to |A|$ is defined by
\begin{equation}\label{eq:s5-div}
 \operatorname{div}(a,b)=|x_0 d_{x_0}^Ra+x_1 d_{x_1}^Rb|.
\end{equation}
It is a $(\operatorname{tDer}(A),{-,-})$ $1$-cocycle with values in $|A|$, namely
\begin{equation}\label{eq:s5-div-cocycle}
 \operatorname{div}\{u,v\}
   =\rho(u)\operatorname{div}(v)-\rho(v)\operatorname{div}(u).
\end{equation}

For $(a,b)\in \mathfrak{sder}_2$, it is \emph{normalized} if the coefficients of $x_0$ in
$a$ and of $x_1$ in $b$ vanish.

\begin{Def}[\cite{Alekseev2012}]
For a normalized $(a,b)\in \mathfrak{sder}_2$, it is in $\mathfrak{krv}_2$, if it satisfies
\[\qquad
 \operatorname{div}(u)=|f(x_0+x_1)-f(x_0)-f(x_1)|,\quad f\in t^2k[[t]].
\]
\end{Def}

By the cocycle property of $\mathrm{div}$, $\mathfrak{krv}_2$ is a Lie subalgebra of $\mathfrak{sder}_2$. There is also the Lie algebra involution $\Theta^{\sigma}$ on $\mathfrak{krv}_2$ induced by the permutation $\sigma(x_0)=x_1$, $\sigma(x_1)=x_0$, 
\begin{equation*}
   \Theta^{\sigma}(a,b)=\left(\sigma(b),\sigma(a) \right)
\end{equation*}
and
$\mathfrak{krv}^{\mathrm{sym}}_2$ denotes the fixed subalgebra under this involution.

We define the map 
\begin{align*}
\mathrm{sd}: L^{\mathrm{sym}} \to \mathfrak{sder}_2,\quad \varphi\mapsto (\varphi(x_{\infty},x_0),\varphi(x_{\infty},x_1)),
\end{align*}
where
$L^{\mathrm{sym}}:=\{\varphi\in L^{\geq3}:
 [x_0,\varphi(x_{\infty},x_0)]+[x_1,\varphi(x_{\infty},x_1)]=0\}$,
which is a Lie algebra with the Lie bracket
\[
 \{\alpha,\beta\}_{I}=[\alpha,\beta]+D_\alpha\beta-D_\beta\alpha,
 \qquad D_\alpha(x_0)=0,\quad D_\alpha(x_1)=[x_1,\alpha],
\]
the map $\mathrm{sd}$ identifies this Lie algebra with the Lie algebra
$\left((\mathfrak{sder}_2^{\mathrm{sym}})_{\geq3},\{-,-\}\right)$.
The inverse sends $(a,b)$ to $\varphi(x_0,x_1)=b(-x_0-x_1,x_1)$.

The Lie algebra $(|A|/\K\cdot 1, [-,-]_{\mathrm{necklace}})$ is a Lie bialgebra with the necklace cobracket $\delta_{\mathrm{necklace}}$ introduced in \cite{Schedler2005}, more explicitly $(|A|/\K \cdot 1,[-,-]_{\mathrm{necklace}},\delta^{red}_{\mathrm{necklace}})$ corresponds to the Lie bialgebra associated with the star-shape quiver in \cite{Schedler2005}. There is the map $\delta_{\mathrm{neclace}}:|A|\to |A|\wedge|A|$ such that $\delta(1)=0$. This map lift $\delta^{red}_{\mathrm neclace}$ and is used in \cite{GTgenus0}.

The shuffle coproduct and antipode induce 
\[
 \widetilde\Delta:|A|\longrightarrow |A|\otimes|A|,
 \qquad
 \widetilde\Delta(|a|)
 =(|-|\otimes|-|)(\id\otimes S)\Delta(a).
\]

\begin{Prop}[\cite{GTgenus0}, Proposition 3.5, Lemma 8.3]\label{Prop:kv_right}

The following diagram is commutative. The map
    $h$ and $\widetilde{\Delta}$ is a $(\mathfrak{sder}_2,[-,-])$ Lie algebra module map, $\tilde{\Delta}$ is injective.
    
    \begin{equation*}
    \begin{tikzcd}
     \mathfrak{sder}_2\arrow{r}{h}\arrow{d}{\mathrm{div}}& \mid A \mid/\K 1 \arrow{d}{-\delta_{\mathrm{necklace}}}\\
      \mid A \mid\arrow{r}{\widetilde{\Delta}}&\mid A\mid\otimes \mid A\mid        
    \end{tikzcd}
    \end{equation*}
\end{Prop}

\subsection{Relating pentagon defect and divergence}
In this susbsection, we recall how to relate the pentagon defect with divergence map developed in \cite{howarth}.

The reduced coaction $\mu:A\to A$ is a linear map defined by
$\mu(1)=\mu(x_0)=\mu(x_1)=0$ and
\[
 \mu(s_1\cdots s_r)=\sum_{j=1}^{r-1}
 s_1\cdots s_{j-1}(s_j\odot s_{j+1})s_{j+2}\cdots s_r,
 \qquad x_i\odot x_j=\delta_{ij}x_i.
\]
It satisfies
\begin{equation}\label{eq:s5-product}
 \mu(ab)=\mu(a)b+a\mu(b)+
 \sum_{i=0}^1(d_{x_i}^La)x_i(d_{x_i}^Rb).
\end{equation}

\begin{Lemma}[{\cite[Lemma 57]{howarth}}]\label{lem:div_mu}
For a homogeneous $u=(a,b)$ in $\mathfrak{sder}_2$  of component degree $m\geq1$,
\[
 |\mu(x_0a+x_1b)|=m\operatorname{div}(u).
\]
Equivalently, if $c\in|A|$ is homogeneous of degree $r\geq2$, then
$\operatorname{div}(H(c))=-|\mu(N(c))|/(r-1)$.
\end{Lemma}

Let $\mathfrak d=\dk_{2,2}(D)$ and
$J=\brun_{2,2}(D)=\ker\pr^3_4\cap\ker\pr^4_4$ and we have $\ker\pr^4_4=\mathbb {L}(t_{14},t_{24},t_{34})$,
$J$ is the ideal generated by $t_{34}$.

\begin{Lemma}\cite{Kuno2025}
    The following is an isomorphism of vector space, $S$ is the antipode, 
    \begin{equation*}
       \mathrm{La}:\quad J/[J,J]\simeq A, \quad f(\ad_{t_{14}},\ad_{t_{24}})t_{34}\mapsto -S(f(x_0,x_1)),\quad f\in A
    \end{equation*}
\end{Lemma}

\begin{proof}
By the Lazad elimination theorem, $J/[J,J]\simeq k\langle t_{14},t_{24}\rangle \otimes t_{34}\simeq A$, the first isomorphism identify the element $f(t_{14},t_{24})t_{34}$ to $\ad_{f(t_{14},t_{24})}t_{34}$, the second one identify the $f(t_{14},t_{24})t_{34}$ with $-S(f(x_0,x_1))$  in $A$ because $S$ is involutive.  
\end{proof}

The map $\mathrm{La}$ is a left $A$ module map, the $A$ module structure on $J/[J,J]$ is through the isomorphism and left multiplication $A$ module structure on $A$. This left module isomorphism is also same as the bimodule structure maps used in \cite{howarth}. 

\begin{Lemma}[\cite{howarth}]\label{lem:pent_J}
For $q\in L^{\geq2}$, set
\[
 r_q(t):=(d_{x_1}^Rq)(t,0)
       =\sum_{n\geq1}c_{x_1x_0^n}(q)t^n.
\]
Then $\pt(q(t_{13},t_{23}))\in J$, 
\begin{equation}\label{eq:s5-pentagon}
 \mathrm{La}\!\left(
 \overline{\pt(q(t_{13},t_{23}))}\right)
 =-\mu(q)-d_{x_1}^Rq-r_q(x_1)+r_q(x_0+x_1).
\end{equation}

\end{Lemma}

For $\varphi\in L^{\geq3}$, put
\[
 q_\varphi=\varphi(x_\infty,x_1),\qquad
 P_\varphi=\mathrm{La}\bigl(\pt^1_{krv}(\varphi)\bigr),\qquad
 r_\varphi=r_{q_\varphi},
 \qquad x_\infty=-x_0-x_1.
\]
The centrality of $t_{12}+t_{13}+t_{23}$ in $\mathfrak t_3$ and
the absence of linear terms in $\varphi$ imply
\[
 \varphi(t_{12},t_{23})
 =\varphi(-t_{13}-t_{23},t_{23})
 =q_\varphi(t_{13},t_{23}).
\]
Thus Lemma~\ref{lem:pent_J} gives
\begin{equation}\label{eq:s5-P-phi}
 P_\varphi=-\mu(q_\varphi)-d_{x_1}^Rq_\varphi
           -r_\varphi(x_1)+r_\varphi(x_0+x_1).
\end{equation}

Similarly to the $W_I$ in \eqref{eq:W_I}, we introduce the space
\[
W_J:=\{f(\ad_{t_{24}})\overline{t_{34}}:
                     f\in\K[[t]]\}\subset J/[J,J].
\]
Our normalization gives
$\mathrm{La}(f(\ad_{t_{24}})\overline{t_{34}})=-f(-x_1)$.
Consequently,
\begin{equation}\label{eq:s5-boundary}
 \begin{aligned}
 \overline{q_\varphi(t_{24},t_{34})}
 &=r_\varphi(-\ad_{t_{24}})\overline{t_{34}},\\
 \mathrm{La}\!\left(\overline{q_\varphi(t_{24},t_{34})}\right)
 &=-r_\varphi(x_1).
 \end{aligned}
\end{equation}
Identifying $W_J$ with $\K[[t]]$ by
$f(-\ad_{t_{24}})\overline{t_{34}}\mapsto f(t)$, we obtain the
boundary coefficient map
\[
 (-)^{\mathrm{ab}}_J:L^{\geq3}\longrightarrow t^2\K[[t]],
 \qquad \varphi\longmapsto r_\varphi.
\]

Let $\sigma$ exchange $x_0$ and $x_1$. For a one-variable series
$r$, set
\[
 \delta_i r=x_i\bigl(r(x_0+x_1)-r(x_i)\bigr),\quad i=0,1,
 \qquad \delta_{\mathrm{sym}}r=\delta_0r+\delta_1r.
\]

We denote
\begin{equation}\label{eq:s5-lifts}
 \begin{aligned}
 \tp_0(\varphi)&=-x_0\sigma(P_\varphi)+\delta_0r_\varphi,\qquad \tp_1(\varphi)=-x_1P_\varphi+\delta_1r_\varphi,\\
 \tp_{\mathrm{sym}}&=\tp_0+\tp_1.
 \end{aligned}
\end{equation}

\begin{Th}\label{th:krv_commutatvie}
Let $\varphi\in L^{\mathrm{sym}}$ be homogeneous of degree $m$,
and write
\[
 \begin{gathered}
 u=\mathrm{sd}(\varphi)=(a,b),\qquad
 a=\sigma(q_\varphi),\quad b=q_\varphi,\\
 T_\varphi=x_0a+x_1b=ax_0+bx_1.
 \end{gathered}
\]
Then
\begin{align}
& \tp_0(\varphi)=\mu(x_0a),\qquad
 \tp_1(\varphi)=\mu(x_1b),\\
& \tp_{\mathrm{sym}}(\varphi)=\mu(x_0a+x_1b)=\mu(ax_0+bx_1)
\end{align}
and
\[
 \operatorname{div}(u)=\frac1m|\tp_{\mathrm{sym}}(\varphi)|.
\]
In particular, with $h(u)=-\frac1{m+1}|x_0a+x_1b|$, the following
diagram commutes degreewise:
\[
\begin{tikzcd}[column sep=large]
 L^{\mathrm{sym}}
   \arrow[r,"\mathrm{sd}"]
   \arrow[d,"\tp_{\mathrm{sym}}"']
 &
 (\mathfrak{sder}_2^{\mathrm{sym}})_{\geq3}
   \arrow[r,"h"]
   \arrow[d,"\operatorname{div}"]
 &
 \lvert A\rvert/\K\cdot1
   \arrow[d,"-\delta_{\mathrm{necklace}}"]
\\
 A
   \arrow[r,"\frac{1}{m}\lvert-\rvert"']
 &
 \lvert A\rvert
   \arrow[r,"\widetilde{\Delta}"']
 &
 \lvert A\rvert\otimes\lvert A\rvert .
\end{tikzcd}
\]
\end{Th}
\begin{proof}
Multiplying \eqref{eq:s5-P-phi} on the left by $-x_1$ gives
\[
 -x_1P_\varphi
 =x_1\mu(b)+x_1d_{x_1}^Rb-\delta_1r_\varphi
 =\mu(x_1b)-\delta_1r_\varphi,
\]
where the second equality is \eqref{eq:s5-product}.
Applying $\sigma$, using $\sigma\mu=\mu\sigma$, gives the
corresponding formula for $a$.
The definitions of the lifts now give the three identities in $A$.
Since $u$ is special, Lemma~\ref{lem:div_mu} gives the divergence
identity. The right square is \eqref{Prop:kv_right}.
\end{proof}

\subsection{The symmetric $\mathfrak{krv}^{\mathrm{sym}}_2$ and $\pt^1_{krv}$}

For general $n,m$, the emergent ideal $c_{m,n}$ of $\dk_{m,n}$ and  the emergent quotient $\mathrm{edk}_{m,n}$ were introduced and studied by D. Bar-Natan and Y. Kuno \cite{Kuno2025}, it is defined as
\begin{equation*}
\mathrm{c}_{m,n}:=(t_{i,j}, m< i,j\le m+n),\quad \mathrm{edk}_{m,n}=\dk_{m,n}/[c_{m,n},c_{m,n}].
\end{equation*}

The Lie algebra $\dk_{2,2}$ is generated by $t_{14},t_{24},t_{34},t_{23},t_{13}$ with the relations,
\begin{align}\label{eq:4T}
&[t_{13},t_{34}]=[t_{34},t_{14}],\quad &[t_{23},t_{34}]=[t_{34},t_{24}],\quad &[t_{13},t_{24}]=0,\\
&[t_{14},t_{13}]=-[t_{14},t_{34}],\quad & [t_{23},t_{24}]=[t_{24},t_{34}],\quad &[t_{14},t_{23}]=0.
\end{align}

In the case of $n=2$, the $ \brun_{m,2}$ coincides with the emergent ideal,
\begin{align*}
    J:=\brun_{2,2}=c_{2,2}=(t_{34}).
\end{align*}

In \cite{Kuno2025}, the Lie algebra $\mathfrak{grt}^{\mathrm{em}}_1$ is introduced and it consists of elements $\psi(x_0,x_1)\in \mathbb{L}(x,y)$ satisfying the condition
\begin{equation}\label{eq:emergent_1}
    \begin{split}
     &\psi(y,0)-\psi(x+y,0)=0,\\
     &d^R_y\psi(x,y)+(d^R_y\psi)(y,0)-(d^R_y\psi)(x+y,0)+\mu(\psi)=0,\\   
    \end{split}
\end{equation}
\begin{equation}\label{eq:emergent_2}
[x,\psi(y,x)]+[y,\psi(x,y)]=0.
\end{equation}

\begin{Th}[Theorem 1.1 in \cite{Kuno2025}]\label{th:Kuno}
The map 
\begin{equation}
\nu^{\mathrm{em}}: \mathfrak{grt}^{\mathrm{em}}_1\to (\mathfrak{krv}^{\mathrm{sym}})_{\ge 2}
\end{equation}
is a graded linear isomorphism
\end{Th}

By the equivalence of the ideal $\brun_{2,2}$ and the emergent ideal $c_{2,2}$ in \cite{Kuno2025}, the following Lemma is a reformulation of the Proposition 4.7 of \cite{Kuno2025} into the notation of $\mathfrak{t_{4}}$.
\begin{Lemma}[\cite{Kuno2025}{Proposition 4.7}]
 The equation \eqref{eq:emergent_1} is equivalent to the equation
 \begin{equation*}
\pt(\psi(t_{13},t_{23}))=0\in \dk_{2,2}/[J,J].    
 \end{equation*}
\end{Lemma}

For an element $\mathbb{L}^{\ge 2}(x,y)$, we have the relation $\varphi(t_{12},t_{23})=\varphi(-t_{13}-t_{23},t_{23})$, let $\psi(t_{13},t_{23})=\varphi(-t_{13}-t_{23},t_{23})$, the equation \eqref{eq:emergent_2} translate into
\begin{equation*}
[x,\varphi(-x-y,x)]+[y,\varphi(-x-y,y)]=0.    
\end{equation*}

As $\dk_{2,2}/[J,J]=\dk_{2,2}/J\oplus J/[J,J]$ and in the following, we show that the part of $\pt(\varphi(-t_{13}-t_{23},t_{23})$ in $\dk/J$ is always zero  for $\psi\in \mathbb{L}(x,y)$, which is another formulation of the Lemma \ref{lem:projection}.

\begin{Prop}
The algebra $U(\dk_{2,2})/(t_{34})$ is isomorphic to $k\langle\langle x,y\rangle\rangle\otimes k\langle\langle x,y\rangle\rangle$, under this isomorphism, for any $\varphi(-t_{13}-t_{23},t_{23})\in U(\dk_{2,2})$ without linear terms and $\varphi(x,0)=\varphi(0,y)=0$, $\varphi$ is a Lie series if and only if $\pt(\varphi(-t_{13}-t_{23},t_{23}))=0\in U(\dk_{2,2})/(t_{34})$
\end{Prop}

\begin{proof}
By checking the 4T relations \eqref{eq:4T} of $\dk_{2,2}$, we have the isomorphisms
\begin{align*}
   U(\dk_{2,2})/(t_{34})&\simeq k\langle\langle x,y\rangle\rangle\otimes k\langle\langle x,y\rangle\rangle,\\
   t_{14},t_{24},t_{13},t_{23}&\mapsto x\otimes 1, y\otimes 1, 1\otimes x, 1\otimes y.
\end{align*}

We then compute the image of $\alpha(\varphi(-t_{13}-t_{23},t_{23}))=0$ under the above isomorphism, it is equal to
\begin{equation*}
\varphi(-t_{13}-t_{14}-t_{23}-t_{24},t_{23}+t_{24})=\varphi(-t_{13}-t_{23},t_{23})+\varphi(-t_{14}-t_{24},t_{24})
\end{equation*}
which is the same as
\begin{equation*}
   \Delta(\varphi(-x-y,x))=\varphi(-x-y,x)\otimes 1+1\otimes\varphi(-x-y,x). 
\end{equation*}
\end{proof}

\begin{Th}\label{th:krv}
Let $\varphi\in \mathbb{L}^{\ge 3}(x,y)$, the element $(\varphi(-x-y,x),\varphi(-x-y,y))$ is in $\mathfrak{krv}^{\mathrm{sym}}_2$ if and only if it satisfies the following two equations
\begin{itemize}
    \item $[x,\varphi(-x-y,x)]+[y,\varphi(-x-y,y)]=0$.
    \item $\pt^1_{krv}(\varphi)=0.$
\end{itemize}
\end{Th}

\begin{Cor}
     Suppose that $[x,\varphi(-x-y,x)]+[y,\varphi(-x-y,y)]=0$, then $\varphi\in \mathfrak{dmr}_0$ implies $((\varphi(-x-y,x),\varphi(-x-y,y))\in \mathfrak{krv}^{\mathrm{sym}}_2$.
\end{Cor}

\section{The \texorpdfstring{$N$}{N}-triviality problem for pentagon braids}
\label{sec:pentagon-triviality}

We study the order of a normalized pentagon braid with respect to the extended $N$ series introduced in the subsection \ref{subsec:extended N series}, namely
  The Brunnian filtration of the free group $F_{m+n-1}$ is the extended N-series
    \begin{align*}
        F_{m+n-1}\ge \mathrm{Brun}_{m,n}(D)\ge \Gamma_2(\mathrm{Brun}_{m,n}(D))\ge \ldots \Gamma_i(\mathrm{Brun}_{m,n}(D))\ge \ldots.
    \end{align*}

    The relative Brunnian filtration of the group $\mathrm{PB}_{m,n}(D)$ is the extended N-series
    \begin{align*}
        \mathrm{PB}_{m,n}(D)\ge \mathrm{Brun}_{m,n}(D)\ge \left(\Gamma_2(\mathrm{PB}_{m,n}(D))\cap \mathrm{Brun}_{m,n}(D)\right)\ge \ldots \left(\Gamma_{k}(\mathrm{PB}_{m,n}(D))\cap \mathrm{Brun}_{m,n}(D)\right)\ge \ldots.
    \end{align*}

The Magnus expansion and the Kontsevich integral detect the order of normalized pentagon braid in the relative Brunnian filtrations. The Johnson--Morita constructions detects the order of pentagon braid in the Brunnian  filtrations. 

Group commutators and conjugation are written
\[
 [a,b]=aba^{-1}b^{-1},\qquad {}^ab=aba^{-1}.
\]
For an action of a group on another group, $[g,a]$ means ${}^ga\,a^{-1}$. We use $\Gamma_1E=E$ and $\Gamma_{r+1}E=[E,\Gamma_rE]$. Composition of morphisms is read from right to left. Statements about groups and Johnson homomorphisms are integral; enveloping algebras and finite-type invariants are over a field $\K$ of characteristic zero. 

\subsection{Normalized pentagon braids and the two filtration orders}
\label{subsec:nt-setup}

Let
\[
 \Phi\in\operatorname{Hom}_{\mathrm{PaB}(3)}((12)3,1(23))
\]
be a morphism whose underlying braid is in $\mathrm{Brun}_3(D)$, namely the following normalization condition,
\begin{equation}\label{eq:nt-normalization}
 d_1\Phi=d_2\Phi=d_3\Phi=1.
\end{equation}

Let $O_{\mathrm L}=(((12)3)4)$ and $O_{\mathrm R}=1(2(34))$, and define two morphisms in $\Hom_{\mathrm{PaB}(4)}(O_{\mathrm L},O_{\mathrm R})$ by
\[
 A_\Phi=\Phi^{2,3,4}\Phi^{1,23,4}\Phi^{1,2,3},\qquad
 B_\Phi=\Phi^{1,2,34}\Phi^{12,3,4}.
\]
We keep the ordered pentagon word
\begin{equation}\label{eq:nt-pentagon-braid}
 \mathrm{Pent}(\Phi)=A_\Phi B_\Phi^{-1}
 =\Phi^{2,3,4}\Phi^{1,23,4}\Phi^{1,2,3}
  (\Phi^{12,3,4})^{-1}(\Phi^{1,2,34})^{-1}.
\end{equation}
The composition is from right to left, therefore it is a loop at $O_{\mathrm R}$. A parenthesized morphism $\Phi$ here is an actual braid, not the formal associator chosen later to construct an expansion.

\begin{Lemma}\label{lem:nt-pentagon-brunnian}
For $\Phi$ satisfying \eqref{eq:nt-normalization}, we have $\mathrm{Pent}(\Phi)$ is in  $\mathrm{Brun}_4(D)$. Moreover if $\Phi\in\Gamma_d\mathrm{PB}_3(D)$, then
\[
 \mathrm{Pent}(\Phi)\in\mathrm{Brun}_4(D)\cap\Gamma_d\mathrm{PB}_4(D).
\]
\end{Lemma}
\begin{proof}
With the assumed normalization property, let $\beta$ denote the underlying braid of $\Phi$, we have the following calculations,
\[
\begin{array}{c|ccccc}
 &\Phi^{2,3,4}&\Phi^{1,23,4}&\Phi^{1,2,3}
 &(\Phi^{12,3,4})^{-1}&(\Phi^{1,2,34})^{-1}\\ \hline
 d_1&\beta&1&1&\beta^{-1}&1\\
 d_2&1&\beta&1&\beta^{-1}&1\\
 d_3&1&\beta&1&1&\beta^{-1}\\
 d_4&1&1&\beta&1&\beta^{-1}
\end{array}
\]
Every row product is the identity, therefore $\mathrm{Pent}(\Phi)$ is in $\mathrm{Brun}_4(D)$. Insertion of a trivial strand and cabling maps are homomorphisms on pure braid groups, therefore preserving the lower central series. Their five images of $\Phi\in\Gamma_d\mathrm{PB}_3(D)$ therefore have product in $\Gamma_d\mathrm{PB}_4(D)$.
\end{proof}

We fix integers $m\ge0$, $n\ge1$ and $N=m+n\ge3$. The first $m$ strands are fixed and the last $n$ are moving. Write
\[
 P=\mathrm{PB}_{m,n}(D),\quad
 H=\mathrm{Brun}_{m,n}(D),\quad
 G=\mathrm{PB}_{m,n-1}(D).
\]
When $n=1$ we set $G=1$. Deletion of the last strand yields a split exact sequence
\begin{equation}\label{eq:nt-last-strand}
 1\longrightarrow F\longrightarrow P\xrightarrow{d_N}G\longrightarrow1,
 \qquad F=F(x_1,\ldots,x_{N-1}),
\end{equation}
put an extral trivial strand to the last is a section with respect to $d_N$, therefore $P=F\rtimes G$. We retain the original labels after deletion. Recall that we have the relations
\begin{equation}\label{eq:nt-H-in-free}
 H=\bigcap_{i\in \{m+1,\ldots,N-1\}}\ker(d_i|_F)\subset F\subset P.
\end{equation}
The groups $F$ and $H$ are subgroups normal in $P$.

\begin{Prop}[The two relative Brunnian filtrations coincide]
\label{prop:nt-ambient}
For every $d\ge1$,
\begin{equation}\label{eq:nt-almost-direct}
 \Gamma_dP=\Gamma_dF\rtimes\Gamma_dG,
 \qquad F\cap\Gamma_dP=\Gamma_dF.
\end{equation}
Therefore we have the relative Brunnian filtration is
\begin{equation}\label{eq:nt-W}
 W_dH:=H\cap\Gamma_dF=H\cap\Gamma_dP.
\end{equation}
The quotients $\Gamma_dP/\Gamma_{d+1}P$ are torsion-free, and
$\bigcap_{d\ge1}\Gamma_dP=\{1\}$.
\end{Prop}
\begin{proof}
The action of $G$ on $F_{\mathrm{ab}}$ is trivial, since pure braids send
meridians to conjugates of themselves. Thus the proposition follows from
the lower-central-series theorem for almost-direct products
\cite{FalkRandell}, iterated along the last-strand decomposition; see also
\cite[Section~4]{MostovyWillerton}.    
\end{proof}

We use the standard generators $x_i=A_{iN}$ of $F$ and put
$S=\{m+1,\ldots,N-1\}$. The Brunnian filtration is
\begin{equation}\label{eq:nt-intrinsic-series}
 K_0=F,\qquad K_r=\Gamma_rH\quad(r\ge1).
\end{equation}
For $h\in H\setminus\{1\}$, define
\[
 \operatorname{ord}_W(h)=\max\{d:h\in W_dH\},
 \qquad
 \operatorname{ord}_H(h)=\max\{r:h\in\Gamma_rH\}.
\]
These orders are finite because $F$ and its free subgroup $H$ are
residually nilpotent. Both orders of the identity are $+\infty$.
For $m+n=4$, the normalized pentagon belongs to $H$, so both orders
apply to $\mathrm{Pent}(\Phi)$.

\subsection{Magnus expansion}
\label{subsec:nt-magnus}

Let
\[
 \mathscr A_{\mathbb Z}
 =\mathbb Z\langle\!\langle X_1,\ldots,X_{N-1}\rangle\!\rangle,
 \qquad
 \mathscr A_{\mathbb Z}^{\ge d}
 =\prod_{j\ge d}(\mathscr A_{\mathbb Z})_j.
\]
The Magnus expansion is the group homomorphism
\begin{equation}\label{eq:nt-magnus}
 \mathcal M:F\longrightarrow1+\mathscr A_{\mathbb Z}^{\ge1},
 \qquad x_i\longmapsto1+X_i.
\end{equation}
We recall the Magnus embedding and the free-group dimension-subgroup
property; see \cite[Sections~3.1 and~4]{MostovyWillerton}.

\begin{Lemma}[Magnus detection]\label{lem:nt-magnus}
The map $\mathcal M$ is injective, and for every $d\ge1$,
\begin{equation}\label{eq:nt-magnus-depth}
 h\in\Gamma_dF
 \quad\Longleftrightarrow\quad
 \mathcal M(h)-1\in\mathscr A_{\mathbb Z}^{\ge d}.
\end{equation}
For $h\in\Gamma_dF$, the degree-$d$ term of $\mathcal M(h)-1$
represents its class under the Magnus--Witt identification
\[
 \Gamma_dF/\Gamma_{d+1}F
 \cong\mathbb L_{\mathbb Z,d}(X_1,\ldots,X_{N-1}).
\]
\end{Lemma}

For $i<N$, let $F^{\widehat i}$ be the free group on the surviving
labelled generators and let $\mathcal M_i$ be its Magnus expansion.
Define
\[
 D_i:\mathscr A_{\mathbb Z}\longrightarrow
 \mathbb Z\langle\!\langle X_j:j\ne i\rangle\!\rangle,
 \qquad D_i(X_i)=0,\quad D_i(X_j)=X_j\ (j\ne i).
\]
Checking the free generators gives
\begin{equation}\label{eq:nt-magnus-deletion}
 D_i\circ\mathcal M=\mathcal M_i\circ d_i.
\end{equation}

\begin{Prop}\label{prop:Magnus_expansion}
For $h\in F$,
\begin{equation}\label{eq:nt-full-brunnian-test}
 h\in H
 \quad\Longleftrightarrow\quad
 D_i(\mathcal M(h)-1)=0\quad\text{for every }i\in S.
\end{equation}
For $h\in H\setminus\{1\}$,
\[
 \operatorname{ord}_W(h)=d
 \quad\Longleftrightarrow\quad
 \mathcal M(h)=1+\phi_d+\text{terms of degree }>d,
 \qquad \phi_d\ne0.
\]
Here $\phi_d$ is a homogeneous Lie polynomial satisfying
$D_i(\phi_d)=0$ for every $i\in S$.
\end{Prop}
\begin{proof}
By \eqref{eq:nt-magnus-deletion} and injectivity of $\mathcal M_i$,
the equations in \eqref{eq:nt-full-brunnian-test} are equivalent to
$d_i h=1$ for all $i\in S$. The remaining assertions follow from
Lemma~\ref{lem:nt-magnus}, since each $D_i$ preserves degree.
\end{proof}

Every nonzero word of $\mathcal M(h)-1$, for $h\in H$, contains each
$X_i$ with $i\in S$. In particular,
\begin{equation}\label{eq:nt-minimum-weight}
 H\subseteq\Gamma_\nu F,
 \qquad\Gamma_rH\subseteq W_{r\nu}H,
 \qquad\nu=\max\{1,n-1\}.
\end{equation}

\subsection{Kontsevich integral and relative Brunnian order}
\label{subsec:nt-kontsevich}

Put
\[
 \mathfrak d=\mathfrak{dk}_{m,n}(D),\qquad
 \mathfrak h=\mathfrak{brun}_{m,n}(D)
 =\bigcap_{i=m+1}^{N}\ker\operatorname{pr}_N^i
 \subset\mathfrak d.
\]
The moving-strand deletions induce
$\operatorname{pr}_N^i:\mathfrak d\to\mathfrak{dk}_{m,n-1}(D)$.
We use the associated-graded identifications established earlier,
\begin{equation}\label{eq:nt-relative-gr}
 \operatorname{gr}_{\K}P\cong\mathfrak d,
 \qquad
 \operatorname{gr}^R_{\K}H
 :=\bigoplus_{d\ge1}(W_dH/W_{d+1}H)\otimes_{\mathbb Z}\K
 \cong\mathfrak h.
\end{equation}

A hat denotes completion in total degree. Fix compatible parenthesized
base objects and a homomorphic expansion induced by a Drinfeld
associator,
\begin{equation}\label{eq:nt-Z}
 \mathcal Z:P\longrightarrow\exp(\widehat{\mathfrak d})
 \subset\widehat U(\mathfrak d).
\end{equation}
By \cite[Definition~4.3 and Proposition~4.4]{Kuno2025}, this expansion is
group-like, preserves filtrations, induces the canonical
associated-graded map, and respects the operadic operations. In
particular,
\begin{equation}\label{eq:nt-Z-deletion}
 \widehat U(\operatorname{pr}_N^i)\mathcal Z
 =\mathcal Z^{\widehat i}d_i,
 \qquad m+1\le i\le N.
\end{equation}

\begin{Prop}\label{prop:Kontsevich_integral}
For $h\in H\setminus\{1\}$,
\[
 \operatorname{ord}_W(h)=d
 \quad\Longleftrightarrow\quad
 \mathcal Z(h)=1+\phi_d+\text{terms of degree }>d,
 \qquad\phi_d\ne0.
\]
The term $\phi_d$ belongs to $\mathfrak h_d$ and is primitive. Under
$X_j\leftrightarrow t_{jN}$, it agrees with the degree-$d$ term of
$\mathcal M(h)-1$ after scalar extension. It is also the degree-$d$
term of $\log\mathcal Z(h)$.
\end{Prop}
\begin{proof}
The filtration-preserving expansion induces an isomorphism on the
associated graded. Since the lower central quotients of $P$ are
torsion-free, it detects their nonzero integral classes. Group-likeness
makes the initial term primitive, and \eqref{eq:nt-Z-deletion} places
it in $\mathfrak h_d$. Both expansions represent the same initial
lower-central class.
\end{proof}

Let $\mathfrak a_P\subset\K P$ be the augmentation ideal. We call
$v:P\to\K$ a finite-type invariant of degree at most $q$ if its linear
extension vanishes on $\mathfrak a_P^{q+1}$. This is the augmentation
formulation of finite-type invariants of pure braids
\cite[Section~1]{MostovyWillerton}; the mixed-braid filtration is
compatible with the expansion \eqref{eq:nt-Z}
\cite[Section~4]{Kuno2025}.

\begin{Prop}[Finite-type triviality]\label{prop:nt-finite-type}
For every $b\in P$ and $d\ge1$,
\begin{equation}\label{eq:nt-vassiliev-test}
 b\in\Gamma_dP
 \quad\Longleftrightarrow\quad
 v(b)=v(1)\text{ for every finite-type invariant of degree }<d.
\end{equation}
In particular, these invariants detect $W_dH$ on $H$.
\end{Prop}
\begin{proof}
For pure braids, this is \cite[Theorem~7]{MostovyWillerton}.
The same conclusion for $P$ follows from
$\Gamma_dP-1\subseteq\mathfrak a_P^d$ and \eqref{eq:nt-Z}:
a functional detecting the first nonzero coefficient
$\mathcal Z_j(b)$ gives a finite-type invariant of degree at most $j$.
The associated-graded isomorphism and torsion-freeness detect every
nonzero lower-central class.
\end{proof}

\begin{Rem}\label{rem:nt-linear-pentagon}
Using the logarithm, the above results can also be used to deteck the linear pentagon equations that we studied in the previous sections. More explicitly, let the underlying braid $\beta$ of $\Phi$ belong to
$\Gamma_d\mathrm{PB}_3(D)$, where $d\ge2$, and let
$\varphi\in\mathbb L_d(x_0,x_1)$ be its degree-$d$ class, with
$x_0=t_{12}$ and $x_1=t_{23}$. Passing the five pentagon factors to
the associated graded gives
\begin{equation}\label{eq:nt-initial-pentagon}
 \bigl(\log\mathcal Z(\mathrm{Pent}(\Phi))\bigr)_d
 =\bigl(\mathcal Z(\mathrm{Pent}(\Phi))-1\bigr)_d
 =\operatorname{pent}(\varphi).
\end{equation}

\end{Rem}

\subsection{Johnson--Morita theory for the Brunnian pair}
\label{subsec:nt-johnson}

\subsubsection{The Brunnian filtration and the acting groups}

The subgroup $H$ is nontrivial: if $S\ne\emptyset$, use an iterated
commutator containing every $x_i$, $i\in S$, or $x_i$ itself when
$|S|=1$; if $S=\emptyset$, then $H=F$. Since $H$ is a nontrivial normal
subgroup of the nonabelian free group $F$, it is nonabelian free.
Put
\[
 Q=F/H,\qquad V=H_{\mathrm{ab}}.
\]
Conjugation induces an action of $Q$ on $V$.

The series $K_*$ in \eqref{eq:nt-intrinsic-series} is the extended
$N$-series determined by the pair $(F,H)$
\cite[Section~10.1]{JM}. The base group $G$ acts through the chosen
section of \eqref{eq:nt-last-strand}, and $P$ acts by conjugation.
Both actions preserve $K_*$, since $H$ is normal in $P$ and each
$\Gamma_rH$ is characteristic in $H$.

\begin{Lemma}\label{lem:nt-deck-faithful}
The quotient $Q$ is torsion-free, and its conjugation action on $V$
is faithful.
\end{Lemma}
\begin{proof}
The deletions embed $Q$ into $\prod_{i\in S}F^{\widehat i}$; for
$S=\emptyset$, $Q=1$. The covering graph of the rose corresponding to
$H$ gives a $Q$-equivariant inclusion
\begin{equation}\label{eq:nt-relation-module}
 V=H_1(\widetilde R;\mathbb Z)
 \lhook\joinrel\longrightarrow C_1(\widetilde R;\mathbb Z)
 \cong\mathbb Z[Q]^{N-1}.
\end{equation}
A nonidentity element of the torsion-free group $Q$ cannot fix a nonzero
finitely supported vector under left translation. Since $V\ne0$,
the action is faithful.
\end{proof}

Let $\mathcal G$ be either $G$ or $P$, acting as above, and put
$\mathcal G_0=\mathcal G$. For $r\ge1$, define
\begin{align}
 \mathcal F_r^{K_*}(\mathcal G)
 &=\{g\in\mathcal G:[g,K_j]\subseteq K_{r+j}
                    \text{ for every }j\ge0\},\label{eq:nt-JM}\\
 \mathcal G_r^0
 &=\ker\bigl(\mathcal G\to\operatorname{Aut}(F/\Gamma_rH)\bigr),
 \label{eq:nt-G0}\\
 \mathcal G_r^1
 &=\ker\bigl(\mathcal G\to
              \operatorname{Aut}(H/\Gamma_{r+1}H)\bigr),
 \label{eq:nt-G1}\\
 \mathcal G_r&=\mathcal G_r^0\cap\mathcal G_r^1.
 \label{eq:nt-Gboth}
\end{align}
For the full conjugation action, write
\begin{equation}\label{eq:nt-JP}
 \mathcal J_r(P)=\mathcal F_r^{K_*}(P).
\end{equation}

\begin{Prop}\label{Prop:identi_4_filtrations}
For every $r\ge1$,
\begin{equation}\label{eq:nt-filtration-identities}
 \mathcal F_r^{K_*}(\mathcal G)
 =\mathcal G_r=\mathcal G_r^1\subseteq\mathcal G_r^0,
 \qquad
 \mathcal G_r^0\supseteq\mathcal G_r\supseteq\mathcal G_{r+1}^0.
\end{equation}
In particular, $\mathcal G_*$ is an extended $N$-series acting on $K_*$.
\end{Prop}
\begin{proof}
The equality $\mathcal F_r^{K_*}(\mathcal G)=\mathcal G_r$ and the
extended $N$-series property are \cite[Propositions~3.1 and~10.1]{JM}.
For the remaining equality use the free-group argument of
\cite[Theorem~10.2]{JM}. At $r=1$, naturality of the $Q$-action and
Lemma~\ref{lem:nt-deck-faithful} show that an automorphism acting
trivially on $V$ also acts trivially on $Q$. The induction for $r\ge2$
is exactly the centralizer argument in that theorem. The final
inclusion follows from $\mathcal G_{r+1}^0\subseteq\mathcal G_r^1$.
\end{proof}

\subsubsection{The associated extended graded Lie algebra}

The associated eg-Lie algebra $L_\bullet=\operatorname{gr}_\bullet(K_*)$
has
\begin{equation}\label{eq:nt-eg-L}
 L_0=Q,\qquad
 L_j=\Gamma_jH/\Gamma_{j+1}H\quad(j\ge1).
\end{equation}
Its bracket and degree-zero action are induced by commutators and
conjugation:
\begin{align}
 [aK_{i+1},bK_{j+1}]&=[a,b]K_{i+j+1}\quad(i,j\ge1),
 \label{eq:nt-eg-bracket}\\
 (fH)\cdot(hK_{j+1})&=fhf^{-1}K_{j+1}\quad(j\ge1).
 \label{eq:nt-eg-action}
\end{align}
Since $H$ is free, the Magnus--Witt theorem gives
\begin{equation}\label{eq:nt-free-L}
 L_+=\bigoplus_{j\ge1}L_j\cong\mathbb L_{\mathbb Z}(V),
 \qquad V=L_1;
\end{equation}
see \cite[Section~10.1]{JM}. No finite-rank assumption on $V$ is needed.

We use the two-component description of extended derivations from
\cite[Section~7]{JM}. Set
\begin{equation}\label{eq:nt-D0}
 \begin{split}
 D_0(L_\bullet)=\{(\alpha,\beta)\in
 \operatorname{Aut}(Q)\times\operatorname{Aut}_{\mathbb Z}(V):\qquad\qquad\qquad\quad\\
 \beta(q\cdot v)=\alpha(q)\cdot\beta(v)
 \text{ for }q\in Q,\ v\in V\}.
 \end{split}
\end{equation}
For $r\ge1$, define
\begin{equation}\label{eq:nt-Dr}
 \begin{split}
 D_r(L_\bullet)=\{(c,\delta)\in
 Z^1(Q,L_r)\times\operatorname{Hom}_{\mathbb Z}(V,L_{r+1}):\qquad\quad\\
 \delta(q\cdot v)=[c(q),q\cdot v]+q\cdot\delta(v)
 \text{ for }q\in Q,\ v\in V\},
 \end{split}
\end{equation}
where
\[
 Z^1(Q,L_r)=\{c:Q\to L_r:c(qq')=c(q)+q\cdot c(q')\}.
\]
By \cite[Proposition~7.3]{JM}, restriction to degrees zero and one gives
\begin{equation}\label{eq:nt-truncation}
 t_r:\operatorname{Der}_r(L_\bullet)\xrightarrow{\sim}D_r(L_\bullet),
 \qquad(d_i)_{i\ge0}\longmapsto(d_0,d_1).
\end{equation}
In degree zero, $\operatorname{Der}_0(L_\bullet)$ means
$\operatorname{Aut}(L_\bullet)$. For $r\ge1$, $d_+$ is the unique
Lie derivation extending $d_1$, and
\begin{equation}\label{eq:nt-eg-compatibility}
 d_i(q\cdot v)=[d_0(q),q\cdot v]+q\cdot d_i(v),
 \qquad v\in L_i,\quad i\ge1.
\end{equation}
We equip $D_\bullet(L_\bullet)$ with the eg-Lie structure of
\cite[Proposition~7.4]{JM}. For degrees $r,s\ge1$, its bracket is
\begin{equation}\label{eq:nt-der-bracket}
 \begin{aligned}
 [d,e]_0(q)&=d_s(e_0(q))-e_r(d_0(q))-[d_0(q),e_0(q)],\\
 [d,e]_1(v)&=d_{s+1}(e_1(v))-e_{r+1}(d_1(v)),
 \end{aligned}
\end{equation}
where the higher components are the unique derivation extensions.
An element $\alpha\in\operatorname{Aut}(L_\bullet)$ acts by
$({}^{\alpha}d)_i=\alpha_{r+i}d_i\alpha_i^{-1}$.

\subsubsection{Johnson homomorphisms}

For $g\in\mathcal G$, let $\tau_0(g)$ be its action on $L_\bullet$.
For $r\ge1$ and $g\in\mathcal G_r$, define
\begin{equation}\label{eq:nt-tau-family}
 \tau_r(g)_i(aK_{i+1})=[g,a]K_{r+i+1},\qquad i\ge0.
\end{equation}
Its two components are
\begin{align}
 \tau_r^0(g)(fH)&=[g,f]\bmod\Gamma_{r+1}H\in L_r,
 \label{eq:nt-tau0}\\
 \tau_r^1(g)(h\Gamma_2H)&=[g,h]\bmod\Gamma_{r+2}H\in L_{r+1}.
 \label{eq:nt-tau1}
\end{align}

\begin{Th}[Johnson--Morita morphism]
\label{prop:Johnson_homo}
The maps $\tau_0$ and $\tau_r$ are homomorphisms with
\[
 \tau_0:\mathcal G\to\operatorname{Aut}(L_\bullet),
 \qquad
 t_r\tau_r=(\tau_r^0,\tau_r^1):\mathcal G_r\to D_r(L_\bullet).
\]
Their kernels are
\begin{equation}\label{eq:nt-tau-kernels}
 \begin{gathered}
 \ker\tau_0=\mathcal G_1,\qquad
 \ker\tau_r=\mathcal G_{r+1},\\
 \ker\tau_r^0=\mathcal G_{r+1}^0,\qquad
 \ker\tau_r^1=\mathcal G_{r+1}\quad(r\ge1).
 \end{gathered}
\end{equation}
They induce an injective eg-Lie algebra morphism
\begin{equation}\label{eq:nt-Johnson-morphism}
 t_\bullet\overline\tau_\bullet:
 \operatorname{gr}_\bullet(\mathcal G_*)
 \lhook\joinrel\longrightarrow D_\bullet(L_\bullet),
 \qquad
 \operatorname{gr}_0(\mathcal G_*)=\mathcal G/\mathcal G_1.
\end{equation}
\end{Th}
\begin{proof}
The existence, kernel of $\tau_r$, and eg-Lie compatibility are
\cite[Propositions~6.1--6.3 and Theorem~6.4]{JM}; the target is
identified by \eqref{eq:nt-truncation}. The individual component
kernels follow from \eqref{eq:nt-tau0}--\eqref{eq:nt-tau1} and
Proposition~\ref{Prop:identi_4_filtrations}.
\end{proof}

The compatibility in \eqref{eq:nt-Dr} reads
\[
 \tau_r^1(g)(q\cdot v)-q\cdot\tau_r^1(g)(v)
 =[\tau_r^0(g)(q),q\cdot v].
\]
In particular, $\tau_r^1(g)$ is $\mathbb Z[Q]$-linear when
$g\in\mathcal G_{r+1}^0$; compare \cite[Proposition~10.5]{JM}. The degree-zero component has the following extension to the larger
filtration.

\begin{Prop}[{\cite[Proposition~10.6]{JM}}]
\label{prop:nt-crossed-extension}
For $r\ge1$, the formula
\begin{equation}\label{eq:nt-crossed-extension}
 \widetilde\tau_r^0(g)(f)=[g,f]\bmod\Gamma_{r+1}H
\end{equation}
defines a map
\[
 \widetilde\tau_r^0:\mathcal G_r^0\longrightarrow Z^1(F,L_r)
\]
with kernel $\mathcal G_{r+1}^0$, where the $F$-action on $L_r$
factors through $Q$. This map is additive for $r\ge2$; for $r=1$,
\begin{equation}\label{eq:nt-crossed-law}
 \widetilde\tau_1^0(gg')
 =\widetilde\tau_1^0(g)+g_V\widetilde\tau_1^0(g').
\end{equation}
Here $g_V$ acts pointwise on the values. The cocycle
$\widetilde\tau_r^0(g)$ descends from $F$ to $Q$ exactly when
$g\in\mathcal G_r$, and then agrees with $\tau_r^0(g)$.
\end{Prop}

\subsection{Intrinsic order of the pentagon braid}
\label{subsec:nt-inner}

We now regard a normalized pentagon as a conjugating element of
$H\subset F\subset P$. Its projection to the base group $G$ is
the identity.

\begin{Th}[Intrinsic Brunnian order]
\label{th:nt-inner-order}
For every $r\ge1$,
\begin{equation}\label{eq:nt-inner-order}
 F\cap\mathcal J_r(P)=\Gamma_rH.
\end{equation}
For $h\in\Gamma_rH$, let $\overline h\in L_r$ be its class. Then
\begin{equation}\label{eq:nt-inner-symbol}
 t_r\tau_r(h)=
 \left(q\longmapsto\overline h-q\cdot\overline h,
       \quad v\longmapsto[\overline h,v]\right).
\end{equation}
This symbol is nonzero if and only if $h\notin\Gamma_{r+1}H$.
\end{Th}
\begin{proof}
Normality and the lower-central commutator estimates give
$\Gamma_rH\subseteq F\cap\mathcal J_r(P)$. Conversely, if
$f\in F\cap\mathcal J_r(P)$, its conjugation action on $V$ is
trivial. Lemma~\ref{lem:nt-deck-faithful} implies $f\in H$.
The condition $[f,H]\subseteq\Gamma_{r+1}H$ now gives
$f\in\Gamma_rH$ by \cite[Lemma~10.3]{JM}.

For $q=aH$, the class of $[h,a]=h\,{}^a(h^{-1})$ is
$\overline h-q\cdot\overline h$; the positive component is induced
by the Lie bracket. These give \eqref{eq:nt-inner-symbol}.
The nonvanishing assertion follows from
$\ker\tau_r=\mathcal J_{r+1}(P)$ and \eqref{eq:nt-inner-order}.
Compare the ordinary inner-action case in \cite[Remark~10.4]{JM}.
\end{proof}

Combining the three constructions gives the required criteria.

\begin{Cor}\label{cor:nt-final-pentagon}
Let $b=\mathrm{Pent}(\Phi)$ be normalized, and choose
$H=\mathrm{Brun}_{m,n}(D)$ with $m+n=4$. For $b\ne1$,
\begin{align*}
 \operatorname{ord}_W(b)=d
 &\quad\Longleftrightarrow\quad
 \mathcal M(b)-1\text{ begins in degree }d\\
 &\quad\Longleftrightarrow\quad
 \mathcal Z(b)-1\text{ begins in degree }d;\\[2pt]
 \operatorname{ord}_H(b)=r
 &\quad\Longleftrightarrow\quad
 b\in\mathcal J_r(P)\text{ and }\tau_r(b)\ne0.
\end{align*}
\end{Cor}

The Johnson filtration separates elements of $F$. Because,
\[
F\cap\bigcap_{r\ge1}\mathcal J_r(P)
=\bigcap_{r\ge1}\Gamma_rH
=\{1\}.
\]
Thus every nontrivial normalized pentagon braid has a nonzero
Johnson symbol at some finite stage.

\bibliographystyle{abbrv}
\bibliography{double_shuffle}

\end{document}